\documentclass[11pt]{amsart}

\usepackage{amsmath}
\usepackage[font=small]{caption}
\usepackage{amsthm}
\usepackage{amsfonts}
\usepackage[pdftex]{graphicx}
\usepackage[margin=.5cm, format=hang]{subfig}
\usepackage[alphabetic]{amsrefs}
\usepackage{amssymb}
\usepackage{tikz-cd}
\usepackage{pinlabel}
\usepackage{morefloats}
\usepackage{enumitem}
\usepackage{fullpage}
\usepackage{array}
\usepackage{multirow}
\usepackage{sidecap}
\usepackage{hhline}
\usepackage{multirow}
\usepackage{blkarray}
\usepackage{xcolor}
\usepackage{comment}

\theoremstyle{plain}
\newtheorem{theorem}{Theorem}[section]
\newtheorem{proposition}[theorem]{Proposition}
\newtheorem{lemma}[theorem]{Lemma}
\newtheorem{corollary}[theorem]{Corollary}
\newtheorem{conjecture}[theorem]{Conjecture}

\theoremstyle{definition}
\newtheorem{definition}[theorem]{Definition}
\theoremstyle{remark}
\newtheorem{remark}[theorem]{Remark}
\newtheorem{example}[theorem]{Example}

\usepackage{mathtools}

\newcommand{\Z}{\mathbb{Z}}
\newcommand{\CS}{\mathrm{CS}}
\newcommand{\gr}{\operatorname{gr}}

\newcommand{\pt}{{*}}

\author[Aliakbar Daemi]{Aliakbar Daemi}
\thanks{AD was partially supported by NSF Grant DMS-2208181 and NSF FRG Grant DMS-1952762.}
\address{Department of Mathematics and Statistics, Washington University in St. Louis}
\email{adaemi@wustl.edu}

\author[Tye Lidman]{Tye Lidman}
\thanks{TL was partially supported by NSF grant DMS-2105469.}
\address{Department of Mathematics, North Carolina State University, Raleigh, NC 27607}
\email{tlid@math.ncsu.edu}

\author[Mike Miller Eismeier]{Mike Miller Eismeier}
\address{Department of Mathematics, University of Vermont, Burlington, VT 05405}
\email{Mike.Miller-Eismeier@uvm.edu}

\title{Filtered instanton homology and cosmetic surgery}

\begin{document}
\maketitle

\begin{abstract}
The cosmetic surgery conjecture predicts that for a non-trivial knot in the three-sphere, performing two different Dehn surgeries results in distinct oriented three-manifolds.  Hanselman reduced the problem to $\pm 2$ or $\pm 1/n$-surgeries being the only possible cosmetic surgeries.  We remove the case of $\pm 1/n$-surgeries using the Chern--Simons filtration on Floer's original irreducible-only instanton homology, reducing the conjecture to the case of $\pm 2$-surgery on genus $2$ knots with trivial Alexander polynomial.  
We also prove some similar results for surgeries on knots in $S^2 \times S^1$. As key steps in establishing these results, we define invariants of the homeomorphism type of 3-manifolds derived from filtered instanton Floer homology and introduce a new surgery relationship for Floer’s instanton homology.
\end{abstract}

\section{Introduction}\label{sec:intro}

\subsection{Cosmetic surgeries}
For a knot $K$ in $S^3$, let the result of $p/q$-surgery be denoted $S^3_{p/q}(K)$.  If $U$ is the unknot, then $S^3_{p/q}(U) = L(p,-q)$ and hence there are infinitely many different surgeries on $U$ that can produce orientation-preserving diffeomorphic three-manifolds. The {\em cosmetic surgery conjecture} \cite[Conjecture 6.1]{Gordon} (see also \cite[Problem 1.81A]{Kirby}) predicts that the unknot is very special in this regard: 
\begin{conjecture}
Let $K$ be a non-trivial knot in $S^3$.  If $p/q \neq p'/q'$, then $S^3_{p/q}(K)$ and $S^3_{p'/q'}(K)$ are not orientation-preserving diffeomorphic.    
\end{conjecture}   

For notation, we say that $p/q$ and $p'/q'$ are a cosmetic pair.  Note that the orientations here are key.  For example, $S^3_{p/q}(4_1)$ is orientation-reversing diffeomorphic to  $S^3_{-p/q}(4_1)$ for any $p/q$ since the figure-eight knot is amphichiral. As a less tautological example, the trefoil admits the chirally cosmetic surgery $S^3_9(3_1) \cong -S^3_{9/2}(3_1)$ \cite{Mathieu}. 

There has been quite a lot of progress on this conjecture.  For example, $\infty$ can never be part of a cosmetic pair \cite[Theorem 2]{GL}.  A sequence of results using Heegaard Floer homology \cite{WangCosmetic,OzSz,Wu,NiWu,Hanselman} gave increasingly stronger constraints on the potential surgery slopes. In particular, \cite[Theorem 1.2]{NiWu} and \cite[Theorem 2(i)]{Hanselman} combine to give:

\begin{theorem}\label{thm:cosmetic-previous}
If $p/q$, $p'/q'$ are a cosmetic pair for a non-trivial knot $K \subset S^3$, then $p/q = -p'/q'$.  Furthermore, $p/q = \pm 2$ or $\pm 1/n$ for some non-zero integer $n$.   
\end{theorem}

We remark that these arguments using Heegaard Floer homology also give further constraints on the cosmetic surgery slopes given more information about the knot (e.g. its knot Floer homology).  Hanselman verified the cosmetic surgery conjecture for all knots up to $16$ crossings \cite[Theorem 6]{Hanselman}.  There have also been many important advances using other aspects of low-dimensional topology.  Detcherry found various obstructions using quantum invariants, and checked the conjecture for knots up to $17$ crossings \cite[Corollary 1.10]{Detcherry}.  Additionally, work of Futer--Purcell--Schleimer \cite[Theorem 7.29]{FPS} rules out cosmetic pairs for hyperbolic knots roughly whenever the slopes have large enough length. They were then able to use bounds on hyperbolic invariants to verify the cosmetic surgery conjecture for all knots up to $19$ crossings \cite[Theorem 2.10]{FPS2}.  For some other examples of the variety of progress on the cosmetic surgery conjecture, see \cite{ItoBridge,SSPretzel, TaoPrime}.  

There is some upper limit on what the Heegaard Floer homology techniques can give: Ozsv\'ath--Szab\'o observed that the Heegaard Floer homology of the $1$ and $-1$ surgeries on the knot $9_{44}$ are graded isomorphic \cite[Section 9]{OzSz}. In the current article, we incorporate a filtration on instanton Floer homology provided by the values of the Chern--Simons functional to obtain more information.  
\begin{theorem}\label{thm:1/n}
If $K$ is a non-trivial knot in $S^3$, then $S^3_{1/n}(K)$ and $S^3_{-1/n}(K)$ are not orientation-preserving diffeomorphic for any integer $n \neq 0$. Furthermore, they cannot be related by a ribbon homology cobordism.
\end{theorem}  

Recall that a ribbon homology cobordism is a homology cobordism that admits a handle decomposition without any 3-handles \cite{ribbon-hom}.

Combining this theorem with the previous literature and an argument using the Casson invariant (see Theorem~\ref{prop:alexander} below) yields:
\begin{corollary} \label{cor:rem-case}
If a non-trivial knot admits a cosmetic surgery, then the pair of slopes is $\pm 2$, the knot has genus 2, the Alexander polynomial is trivial, and the value of the Jones polynomial at $e^{2\pi i/5}$ is $1$.  In particular, the cosmetic surgery conjecture holds for fibered knots, knots which are not topologically slice, and $HFK$-thin knots.    
\end{corollary}
\begin{proof}
That the only possible pair of cosmetic slopes is $\pm 2$ is a combination of Theorems~\ref{thm:cosmetic-previous} and \ref{thm:1/n}.  Hanselman shows that if $\pm 2$ is a cosmetic pair, then the knot has genus $2$. The claim about the Jones polynomial of $K$ follows from \cite[Theorem 1.4]{Detcherry}. Finally, in Theorem~\ref{prop:alexander} below, we show that in this special case the Alexander polynomial must be trivial.  Since non-trivial knots with Alexander polynomial one are non-fibered, are topologically slice \cite{Freedman-Disk, FreedmanQuinn}, and have thick knot Floer homology \cite{OzSz-HFK, OzSz-Genus}, we have the desired result.
\end{proof}

\begin{remark}
Since any alternating knot is $HFK$-thin, Corollary \ref{cor:rem-case} implies that the cosmetic surgery conjecture holds for alternating knots. For a large family of alternating knots, this was already verified in the recent work \cite{IJ:alt}.
\end{remark}

\begin{remark}\label{rmk:alex-jones}
Dave Futer has reported to the authors that there are in fact no non-trivial knots with at most 17 crossings which have both $J_K(e^{2\pi i/5}) = 1$ and $\Delta_K = 1$. 
We also remark that the results of \cite{Detcherry} can be used to give further constraints for the remaining case of surgery slopes $\pm 2$ in terms of the colored Jones polynomial of $K$.
\end{remark}

\begin{remark}
Combined with the work of Ravelomanana \cite{Ravelomanana}, this proves that there are no exceptional cosmetic surgeries on hyperbolic knots in $S^3$.  Since the first version of this article appeared, Ren used Theorem~\ref{thm:1/n} to reduce the cosmetic surgery conjecture to the case of hyperbolic knots \cite{Ren}.  
\end{remark}

Using similar arguments, we are also able to prove a restricted cosmetic surgery statement in $S^2 \times S^1$.  
\begin{theorem}\label{thm:S2timesS1}
If $L$ is a knot in $S^2 \times S^1$ which is homologous but not isotopic to $\{\pt\} \times S^1$, then no two integral surgeries on $L$ can produce the same oriented three-manifold.  Furthermore, they cannot be related by a ribbon homology cobordism.
\end{theorem}

Note that the three-manifolds described in Theorem~\ref{thm:S2timesS1} are all homology spheres. Furthermore, they are boundaries of  {\it Mazur manifolds} and hence are homology cobordant to $S^3$. (A Mazur manifold is a four-manifold obtained by attaching a 2-handle to $S^1 \times D^3$ along a curve which generates $H_1$.) In particular, the homology spheres from Theorem \ref{thm:S2timesS1} are homology cobordant to each other, but the ones provided by surgeries on the same knot are not related by homology cobordisms without any 3-handles.

A Mazur manifold is $B^4$ if and only if the attaching curve is isotopic to $\{\pt\} \times S^1$ \cite{Gabai-FoliationsII}. Theorem \ref{thm:S2timesS1} has the following related four-dimensional consequence. We are not aware of any other three- or four-dimensional proof of this fact in general.  

\begin{corollary}
If $M$ is a Mazur manifold other than $B^4$, then changing the framing of the 2-handle in the Mazur description results in a four-manifold not oriented homeomorphic to $M$.  
\end{corollary}

As another application of Theorem~\ref{thm:S2timesS1}, we are able to reprove and extend part of a theorem of Josh Wang \cite[Theorem 1.1]{Wang}.
\begin{corollary}\label{cor:josh}
Let $K_0$ be obtained by a non-trivial band surgery on a two-component unlink, and $K_n$ be the result of putting $n$ full twists in the band. Then for any non-zero integer $p$, $S^3_{1/p}(K_n)$ and $S^3_{1/p}(K_m)$ are not orientation-preserving diffeomorphic unless $m = n$.  Consequently, $K_n$ and $K_m$ are distinct knots for any two different values of $m$ and $n$.  
\end{corollary}
\begin{proof}
Fix $p$, and let $\gamma$ denote the meridian of the band, thought of as a knot in $S^3_{1/p}(K_0)$. The manifold $S^3_{1/p}(K_0)$ is not $S^3$ as long as $p \neq 0$ and $K_0$ is non-trivial, which is the case when the band is non-trivial \cite{Scharlemann}.  Then $S^2 \times S^1$ is $0$-surgery on $\gamma$, while $S^3_{1/p}(K_n)$ is the result of $-1/n$-surgery on $\gamma$, or alternatively, integral surgery on the core of $\gamma$ in $S^2 \times S^1$.  The result follows from Theorem~\ref{thm:S2timesS1}.
\end{proof}
Note that the statement is in some ways slightly stronger than Wang's theorem for full twists, as his results do not immediately distinguish the surgered manifolds.  However, his results apply to a broader class of links, namely possibly half-twisted bandings between arbitrary two-component split links.  Wang has proposed a strategy for generalizing Corollary~\ref{cor:josh} to this more general case, but it seems like the requisite technical machinery for instanton Floer homology is not currently developed.

\begin{remark} 
While we do not do this here, similar ideas to the ones in this paper can be applied to prove analogous results in other manifolds. For example, suppose $K$ is a knot in $\mathbb{RP}^3$ which is neither null-homologous nor a core curve of the genus 1 Heegaard splitting. Then the two integer homology spheres which arise as integral surgeries on $K$ are distinct as oriented three-manifolds.  
\end{remark}

\subsection{The strategy and surgery exact triangles in instanton Floer homology}
We now describe our overarching strategy.  For the introduction, we give broad strokes for non-experts. Instanton Floer homology morally assigns to a homology sphere a chain complex over $\mathbb{Z}[t,t^{-1}]$ whose generators corresponds to conjugacy classes of non-trivial $SU(2)$ representations of $\pi_1$. Any monomial generator $\alpha$ gets a real value from the Chern--Simons functional denoted $\CS(\alpha)$ and a $\Bbb Z$-grading denoted by $\gr(\alpha)$. Multiplication by $t$ raises $\gr(\alpha)$ by 8 and $\CS(\alpha)$ by 1.\footnote{For the experts, we consider flat connections modulo degree zero gauge transformations, so Chern--Simons can be viewed as real-valued and the associated Floer homology is $\mathbb{Z}$-graded.  Of course, we also need to introduce perturbations everywhere to obtain the desired transversality results.  We assume in this introduction that all moduli spaces are cut out transversely and so no perturbations are needed.}  Furthermore, the differential has degree $-1$ with respect to $\gr$ and is filtered by the Chern--Simons functional.  Our strategy is to show that the Chern--Simons filtered instanton homologies of different surgeries on knots are different.  The strategy is roughly to find an isomorphism which strictly lowers the filtration. These maps will come from functoriality of instanton Floer complexes with respect to cobordisms. 

A cobordism $W:Y \to Y'$ between integer homology spheres induces a map of instanton Floer homologies by counting ASD connections, given a choice of bundle on $W$.   In the simplest case, if $W$ is a negative-definite cobordism and $A$ is an ASD connection over the trivial $SU(2)$-bundle on $W$ which is asymptotic to monomial generators $\alpha$ at $Y$ and $\alpha'$ at $Y'$, then $\CS(\alpha) \geq \CS(\alpha')$ and $\gr(\alpha) = \gr(\alpha')$.  This inequality is strict if $\pi_1(W) = 0$.  Consequently, if $W$ is simply-connected, negative-definite, and $W: Y \to Y'$ induces an isomorphism on instanton homology, then we get a filtration decreasing isomorphism of Floer homologies. If the instanton Floer homology of $Y$ is non-trivial, this is only possible if the filtered instanton Floer homologies of $Y$ and $Y'$ are different, and so the three-manifolds are not orientation-preserving diffeomorphic to each other. This strategy is sufficient to prove Theorem~\ref{thm:S2timesS1}, which we next describe.  The strategy for Theorem \ref{thm:1/n} is more complicated, and is discussed after the warm-up case of knots in $S^2 \times S^1$.  

\subsubsection{The strategy to prove Theorem~\ref{thm:S2timesS1}}
We will use Floer's original exact triangle for a knot $K$ in a homology sphere $Y$:
\begin{equation}\label{eq:floer-triangle}
\ldots \to I_*(Y) \to I_*(Y_{-1}(K)) \to I_*^w(Y_0(K)) \to \ldots
\end{equation}
where $I_*^w(Y_0(K))$ is the admissible version of instanton homology for three-manifolds with $w$ determining the bundle data \cite{Floer-knot,BD:sur-tr}. If $L\subset S^2\times S^1$ is given as in the statement of Theorem \ref{thm:S2timesS1}, we apply \eqref{eq:floer-triangle} to the case that $Y$ is given by an integral surgery on $L$ and $K\subset Y$ is the dual knot. Then $Y_0(K)=S^2\times S^1$ and $Y_{-1}(K)$ is the integral surgery on $L$ where we increase the surgery coefficient by one. Since $I_*^w(S^2 \times S^1) = 0$, the map $I_*(Y) \to I_*(Y_{-1}(K))$ in \eqref{eq:floer-triangle} is an isomorphism. This map is induced by the 2-handle cobordism $W:Y \to Y_{-1}(K)$, which is negative definite and simply-connected. Furthermore, $I_*(Y)$ is non-zero because $L$ is not isotopic to $\{\pt\} \times S^1$ \cite[Theorem 1.3]{LPCZ}. Thus the observation in the previous paragraph can be applied to $W$ to show that $Y$ and $Y_{-1}(K)$ are not orientation-preserving diffeomorphic to each other. By stacking such cobordisms together, we can more generally show that any two integral surgeries on $K$ are not orientation-preserving diffeomorphic to each other. The proof of the second part of Theorem~\ref{thm:S2timesS1} uses a result from \cite{ribbon-hom} about the behavior of instanton Floer homology with respect to ribbon homology cobordisms to obtain a similar contradiction.

\subsubsection{The strategy to prove Theorem~\ref{thm:1/n} for $n=\pm 1$}
Unfortunately, it is generally impossible to construct a simply-connected, negative definite cobordism from $S^3_{1/n}(K)$ to $S^3_{-1/n}(K)$ whose usual cobordism map is an isomorphism. Such a cobordism map would necessarily have degree zero, but for example $S^3_1(3_1)$ and $S^3_{-1}(3_1)$ have their instanton Floer homology supported in different gradings; we need a variation on the constructions above to make this work.  We begin with $\pm 1$-surgery. 

In order to relate $S^3_1(K)$ and $S^3_{-1}(K)$, we need to establish another surgery exact triangle --- which we call the \emph{distance-two surgery triangle} --- and extract information from it in a novel way. The exact triangle we prove is predicted in \cite{CDX}, who prove an analogous version for Floer homology with admissible bundles:

\begin{theorem}\label{thm:CDX}
Let $K$ be a knot in a homology sphere $Y$.  Then there is an exact triangle of the following form:
\begin{equation}\label{eq:CDX-triangle}
\ldots \to I_*(Y_1(K)) \to I_*(Y) \oplus I_*(Y)  \to I_*(Y_{-1}(K)) \to \ldots
\end{equation}
\end{theorem}

\begin{remark}
A similar exact triangle for the Heegaard Floer homology groups $\widehat{HF}$ is established as \cite[Theorem 3.1]{OS-surgery}. While the groups $\widehat{HF}$ should be understood as analogous to the instanton homology groups $I^\#(Y)$, there is no known analogue of Floer's irreducible instanton homology groups $I_*(Y)$ in Heegaard Floer theory.
\end{remark}

Taking $Y = S^3$, we obtain an isomorphism from $I_*(S^3_{-1}(K))$ to $I_*(S^3_1(K))$ because $I_*(S^3) = 0$.  Unfortunately, the cobordism $W:S^3_{-1}(K) \to S^3_{1}(K)$ inducing the isomorphism has $b^+ > 0$, and it does not behave in the desired way with respect to the Chern--Simons filtration; the existence of this map does not lead to a contradiction.

However, this is not the only cobordism map available to us.  In the proof of the exact triangle, one needs to find a nullhomotopy of the chain map that represents the composition 
\[
  I_*(Y_1(K)) \to I_*(Y) \oplus I_*(Y) \to I_*(Y_{-1}(K)).
\]  
This nullhomotopy comes from a count of ASD connections over a 1-parameter family of metrics on a different cobordism $W'$. In the case of $Y = S^3$, this count turns out to be a chain map $g_1$, and even better, a quasi-isomorphism which is a homotopy inverse of the cobordism map for $W$.  (We emphasize that the 1-parameter family map is different from the usual cobordism map for $W'$, which is identically 0.)  The cobordism $W'$ is negative-definite and if $\alpha_\pm$ on $S^3_{\pm 1}(K)$ are related by an ASD connection over $W'$, then $\CS(\alpha_-) - \gr(\alpha_-)/8 < \CS(\alpha_+) - \gr(\alpha_+)/8$; that is, $W'$ induces a strictly filtered map with respect to a degree-shifted Chern--Simons filtration. The instanton Floer complex is finitely generated over $\Bbb Z[t, t^{-1}]$ and thus, the function $\CS - \gr/8$ is bounded on the collection of all monomial generators. Because $g_1$ is strictly filtered, this once again implies that the filtered Floer homologies of $S^3_1(K)$ and $S^3_{-1}(K)$ are not isomorphic.

This can be explained more succinctly by defining an invariant $\ell(Y)$, which is loosely the minimal value of $\CS - \gr/8$ on a generator of instanton homology.  (See Example~\ref{ex:ell} for the computation of $\ell$ for the Poincar\'e homology sphere.)  The exact triangle \eqref{eq:CDX-triangle} is used to show that $\ell(S^3_{-1}(K)) < \ell(S^3_{1}(K))$.  By further applying \eqref{eq:floer-triangle} and \eqref{eq:CDX-triangle}, we obtain a sequence of inequalities
\begin{equation}\label{eq:ell}
-\infty < \ldots < \ell(S^3_{-1/2}(K)) < \ell(S^3_{-1}(K)) < \ell(S^3_{1}(K)) < \ell(S^3_{1/2}(K)) < \ldots < \infty
\end{equation}
which completes the proof that $S^3_{1/n}(K)$ is not orientation-preserving diffeomorphic to $S^3_{-1/n}(K)$. (The additional inequalities do not require the use of the nullhomotopy maps.) The claim about ribbon homology cobordisms follows as in the proof of Theorem \ref{thm:S2timesS1}. 

\begin{remark}
This argument does not require knowing the theorem of Ni--Wu that $S^3_{p/q}(K) = S^3_{p/q'}(K)$ implies $q = \pm q'$.  
\end{remark}

In \cite[Proposition 1.15]{LLP}, it was shown that the existence of a non-trivial knot $K$ with $S^3_1(K) = S^3_{-1}(K)$ would produce an exotic $S^1 \times S^3 \# S^2 \times S^2$.  The existence of such a knot is now ruled out by Theorem~\ref{thm:1/n}.  Similarly, if there exist non-trivial knots $K_1, \ldots, K_n$ such that $S^3_1(K_i) = S^3_{-1}(K_{i+1})$ for $i = 1,\ldots,n$ with indices computed mod $n$, then there exists an exotic $S^3 \times S^1 \#_n S^2 \times S^2$.  By the inequalities in \eqref{eq:ell}, we see it is also impossible to construct an exotic $S^3 \times S^1 \#_n S^2 \times S^2$ by this scheme. More generally, we obtain the following result.

\begin{theorem}
Let $K_1,\ldots,K_n$ be a sequence of knots and $a_1,\ldots,a_n$ and $b_1,\ldots, b_n$ integers with $a_i > b_{i+1}$ with indices computed mod $n$.  If $S^3_{1/a_i}(K_i) = S^3_{1/b_{i+1}}(K_{i+1})$ for all $i$, then all $K_i$ are unknotted.
\end{theorem}

\begin{remark}
It is natural to wonder whether our arguments can be applied to the case of $\pm 2$ surgery. Even though $S^3_{\pm 2}(K)$ are not integer homology spheres, these still have a well-defined filtered irreducible instanton homology $I_*(S^3_{\pm 2}(K))$, and one can define invariants $\ell(Y_{\pm 2})$. We have $\ell(S^3_{-1}(K)) < \ell(S^3_{-2}(K))$ and $\ell(S^3_2(K)) < \ell(S^3_1(K))$, but our initial investigations suggest that it is more difficult to compare the values $\ell(S^3_{\pm 2}(K))$, even though these manifolds still have isomorphic Floer homology. In addition, while there exist variations on Theorem \ref{thm:CDX} relating the instanton homologies of $\pm 2$ surgery, the variations known to the authors at the time of writing are all inadequate in some way: for one variation, the complex replacing $C_*(S^3) \oplus C_*(S^3)$ is no longer trivial so that the map $g_1$ is no longer a chain map; for another variation, this complex is trivial and $g_1$ is a chain map, but it behaves unfavorably with respect to the Chern--Simons filtration. 
\end{remark}

\section*{Organization}
In Section 2, we review the basics of persistent homology, including a variation in the instanton setting. We then introduce the invariant $\ell$ used above and verify its basic properties. In Section 3, we recall the necessary background material on filtered instanton Floer homology, and state the new results relating to the distance-two exact triangle. In Section 4 we prove Theorems \ref{thm:S2timesS1} and \ref{thm:1/n}, taking for granted the distance-two exact triangle of Theorem \ref{thm:CDX}; we establish the existence of this exact triangle in Section 5, which constitutes about half of the paper. The final Section 6 gives a short proof that if $K$ admits cosmetic surgeries then $\Delta_K(t) = 1$.

\section*{Acknowledgments}
We gladly thank Josh Wang for some critical discussions from which this paper was developed and Peter Kronheimer for pointing out a crucial error in a primitive version of this strategy.  We also thank Dave Futer and Jonathan Hanselman for helpful comments.

\section{Filtered groups and numerical invariants}\label{sec:kappa}

Real-valued lifts of the Chern--Simons functional have been used extensively to define filtered algebraic structures and extract numerical invariants of $3$-manifolds, and these numerical invariants have had extensive topological applications \cite{CS1,CS2,Daemi-Gamma,NST}. The Chern–Simons filtration is fundamental to the results of this paper, and we will use the framework of persistent homology to define the associated algebraic structure.

A \emph{persistence module} $V_\bullet$ is a collection of $R$-modules $F_r V$ indexed by $r \in \Bbb R \cup \{\infty\}$, together with connecting homomorphisms $i_r^{r'}: F_r V \to F_{r'} V$ when $r \le r'$ which are functorial in the sense that $i_{r'}^{r''} i_{r}^{r'} = i_r^{r''}$ \cite{PM1,PM2}. The basic example is the \emph{interval module}, and over a field every persistence module decomposes into these basic pieces:

\begin{example}
Suppose $I \subset \Bbb R \cup \{\infty\}$ is an interval. The \emph{interval module} $R_I$ has $$F_r R_I = \begin{cases} R & r \in I \\ 0 & \text{otherwise}, \end{cases}$$ 
with connecting homomorphisms $i_r^{r'}$ equal to the identity if $r, r' \in I$ and equal to zero otherwise. 
\end{example}

\begin{theorem}[\cite{PM3}]
If $V$ is a persistence module over a field $k$ for which $F_r V$ is finite-dimensional for all $r$, then there exist a collection of intervals $I_1, I_2, \ldots$ and an isomorphism of persistence modules $V \cong \bigoplus_{j} k_{I_j}$. 
\end{theorem}

The collection of intervals in a direct sum decomposition $V_\bullet \cong \bigoplus_{j} k_{I_j}$ is well-defined, and called the \emph{barcode} of $V$. Notice that the statement above does not imply that there are finitely many intervals in the barcode for $V$, only that each $r \in \Bbb R$ lies in finitely many of these intervals. 

In the instanton theory, the filtration has a natural $8$-periodicity, so we introduce the closely related notion of \emph{instanton persistence modules} or $IP$-modules. 

\begin{definition}\label{def:I-module}
An \textbf{$IP$-module} $A_\bullet$ over the ring $R$ consists of the following data.
\begin{enumerate}[label=(\roman*)]
\item For each integer $d$ and each $r \in \Bbb R \cup \{\infty\}$, an $R$-module $F_r A_d$, and an isomorphism $\varphi_{r,d}: F_r A_d \to F_{r+1} A_{d+8}$. 
\item For each integer $d$ and each $r \le r'$, a homomorphism $i_r^{r'}: F_r A_d \to F_{r'} A_d$. 
\end{enumerate}
We demand that these satisfy the following conditions:
\begin{enumerate}[label=(\alph*)]
\item The map $i_r^r$ is the identity, $i_{r'}^{r''} i_{r}^{r'} = i_{r}^{r''}$, and $\varphi_{r', d} i_r^{r'} = i_{r+1}^{r'+1} \varphi_{r,d}$.
\item For each $d$, we have $F_r A_d = 0$ for sufficiently small $r \in \Bbb R$, while the map $i_r^\infty: F_r A_d \to F_\infty A_d$ is an isomorphism for sufficiently large $r \in \Bbb R$.
\end{enumerate}
\end{definition}

\noindent When there is no risk of confusion, we will often write $A_d$ in place of $F_\infty A_d$. 

In the case of an $IP$-module $V_\bullet$ over a field $k$ with each $F_r V_d$ finite-dimensional, so that $V_d \cong \bigoplus_j k_{I_{j,d}}$, the periodicity property is essentially equivalent to the statement that the barcode in degree $d+8$ is a shift of the barcode in degree $d$. More precisely, the number of intervals $I_{j,d}$ and $I_{j,d+8}$ is the same, and each interval $I_{j,d+8}$ of $V_{d+8}$ is the shift $I_{j,d} + 1$ of a corresponding interval for $V_d$.

We will only use fairly crude numerical invariants, which ignore much of the structure of an $IP$-module.

\begin{definition}\label{def:kappa}
If $A_\bullet$ is an $IP$-module, define $\kappa_A: \Bbb Z \to [-\infty, \infty]$ and $\ell(A_\bullet) \in [-\infty, \infty]$ as 
\begin{align*}\kappa_{A_\bullet}(d) &= \inf \{r \in \Bbb R \mid i_r^{\infty} : F_r A_d \to A_d \text{ is nonzero}\}, \\
\ell(A_\bullet) &= \inf \{\kappa_A(d) - d/8 \mid d \in \Bbb Z\}.\end{align*}
\end{definition}

We will usually write $\kappa_A$ in place of $\kappa_{A_\bullet}$ to simplify notation. These are transparently invariant under $IP$-module isomorphisms (isomorphisms $F_r A_d \cong F_r B_d$ which commute with the structure maps). If we work over a field $k$ and $A_\bullet$ is a direct sum of interval modules, then $\kappa_A(d)$ is the least $r$ such that $[r, \infty]$ appears in the barcode of $A_d$, while $\ell(A_\bullet)$ is the least $r$ such that $[r+d/8, \infty]$ appears in the barcode of $A_d$ for some $d$.

\begin{lemma}\label{lemma:kappa-basic}
The $\kappa$ and $\ell$ invariants satisfy the following properties.
\begin{enumerate}[label=(\alph*)]
\item We have $\kappa_A(d) > -\infty$ for all $d$, and $\kappa_A(d) = \infty$ if and only if $A_d = 0$. 
\item We have $\kappa_A(d+8) = \kappa_A(d) + 1$. 
\item We have $\ell(A_\bullet) > -\infty$, and $\ell(A_\bullet) = \infty$ if and only if $A_d = 0$ for all $d$.
\end{enumerate}
\end{lemma}
\begin{proof}
The first claim follows immediately from axiom (b) of an $IP$-module. 
For the second claim, we may assume $\kappa_A(d) = r < \infty$. Suppose $x \in F_{r+\epsilon} A_d$ has $i_{r+\epsilon}^{\infty} (x) = y\ne 0$. Then by axiom (a), we have $$i_{r+1+\epsilon}^{\infty} \varphi_{r+\epsilon, d}(x) = \varphi_{\infty, d} i_{r + \epsilon}^{\infty}(x) = \varphi_{\infty,d}(y) \ne 0.$$ The last inequality holds because $y \ne 0$ and $\varphi_{\infty, d}$ is an isomorphism. Because $\varphi_{r+\epsilon, d}(x) \in F_{r+1+ \epsilon} A_{d+8}$, we have $\kappa_A(d+8) \le \kappa_A(d) + 1 + \epsilon$ for all $\epsilon > 0$, and hence that $\kappa_A(d+8) \le \kappa_A(d) + 1$. The other inequality is similar.
For the last claim, observe that $\kappa_A(d) - d/8$ is an $8$-periodic function by the second claim, and hence takes on only finitely many values. The minimum of these values is infinite if and only if $\kappa_A(d) = \infty$ for all $d$, so the result follows from the first claim.
\end{proof}

Next, we introduce the type of morphisms relevant to us.

\begin{definition}
Suppose $A_\bullet$ and $B_\bullet$ are $IP$-modules. An \textbf{$IP$-morphism $f: A_\bullet \to B_\bullet$ of degree $D$ and level $L$} is a collection of homomorphisms $f_{r,d}: F_r A_d \to F_{r+L} B_{d+D}$ satisfying $$f_{r',d} i_r^{r'} = i_{r+L}^{r'+L} f_{r,d}, \quad \quad f_{r+1,d+8} \varphi_{r,d} = \varphi_{r+L,d+D} f_{r,d}.$$
\end{definition}

\begin{lemma}\label{lemma:kappa-ineq}
Suppose $f: A_\bullet \to B_\bullet$ is an $IP$-morphism of degree $D$ and level $L$. 
\begin{enumerate}[label=(\alph*)]
\item If $f_{\infty, d}: A_d \to B_{d+D}$ is injective, then $\kappa_B(d+D) \le \kappa_A(d) + L$.
\item If $f_{\infty, d}$ is injective for all $d$, then $\ell(B_\bullet) \le \ell(A_\bullet) + (L - D/8).$
\end{enumerate}
\end{lemma}

\begin{proof}
If $\kappa_A(d) = \infty$ the first claim is vacuous, so suppose $\kappa_A(d) = r < \infty$, and that $x \in F_{r+\epsilon} A_d$ has $i_{r +\epsilon}^{\infty}(x) \ne 0$ for $\epsilon > 0$. Setting $y = f_{r+\epsilon, d}(x) \in F_{r+\epsilon+L} B_{d+D}$, we have $$i_{r+\epsilon + L}^{\infty}(y) = f_{\infty, d} i_{r+\epsilon}^{\infty}(x) \ne 0,$$ as $f_{\infty, d}$ is injective. It follows that $\kappa_B(d+D) \le \kappa_A(d) + L + \epsilon$ where $\epsilon$ can be arbitrarily small. This proves the first claim.

As for the second claim, we have 
\begin{align*}\ell(B_\bullet) &= \inf \{\kappa_B(d+D) - d/8 - D/8 \mid d \in \Bbb Z\}  \\
&\le \inf \{\kappa_A(d) - d/8 + L - D/8 \mid d \in \Bbb Z\} = \ell(A_\bullet) + (L-D/8). \qedhere\end{align*}
\end{proof}

\begin{remark}
The above statement also has an interpretation in terms of barcodes when working over a field. If $f_{r',d}$ is injective, and $[r, r']$ is an interval in the barcode of $A_d$, then $[r+L, r'+L]$ is \emph{contained} in an interval in the barcode of $B_{d+D}$. Lemma \ref{lemma:kappa-ineq}(a) is the case $r' = \infty$. 
\end{remark}

We will later need a mild generalization of the above claim. 

\begin{lemma}\label{lemma:inhomog-ineq}
Suppose $A,B$ are $IP$-modules, and $f^i: A_\bullet \to B_\bullet$ is an $IP$-morphism of degree $D_i$ and level $L_i$ for $i = 1, 2$. If the map $(f^1_\infty, f^2_\infty): A_d \to B_{d+D_1} \oplus B_{d+D_2}$ is injective for all $d$, then we have $$\ell(B_\bullet) \le \ell(A_\bullet) + \max(L_1 - D_1/8, L_2 - D_2/8).$$
\end{lemma}
\begin{proof}
If $\ell(A_\bullet) = r$, for all $\epsilon > 0$ there exists some $d$ and $x \in F_{r + \epsilon + d/8} A_d$ so that $i_{r+\epsilon+d/8}^{\infty}(x) \ne 0$. Because $(f^1_\infty, f^2_\infty)$ is injective, for some $i \in \{1,2\}$ we have $$0 \ne f^i_\infty i_{r + \epsilon + d/8}^{\infty}(x) = i_{r+\epsilon+d/8+L_i}^{\infty} f^i_{r+\epsilon+d/8}(x).$$ As $f^i_{r+\epsilon + d/8}(x) \in F_{r + \epsilon + d/8 + L_i} B_{d+D_i}$, this implies that 
\begin{align*}\ell(B_\bullet) &\le \kappa_B(d+D_i) - d/8 - D_i/8 \le (r + \epsilon + d/8 + L_i) - d/8 - D_i/8 \\
&= \ell(A_\bullet) + \epsilon + (L_i - D_i/8) \le \ell(A_\bullet) + \text{max}(L_1 - D_1/8, L_2 - D_2/8) + \epsilon.
\end{align*}
Taking $\epsilon$ to zero completes the argument.
\end{proof}

\section{Instanton Floer homology}\label{sec:instantons}

In this section, we review the aspects of instanton Floer theory necessary to prove our main results. At first, we let $R$ be an arbitrary commutative ring. Later we will specialize to $R = \Bbb F_2$.

\subsection{IP-modules from instantons} Some of the essential properties of instanton Floer homology for our purposes are collected in the following proposition.

\begin{proposition}\label{prop:Floer-def}
Floer's instanton homology groups satisfy the following properties. 
\begin{enumerate}[label=(\alph*)]
\item If $Y$ is an integer homology sphere, there is an associated $\Bbb Z/8$-graded module $I_d(Y;R)$, invariant under orientation-preserving diffeomorphism and functorial under cobordisms $(W,c): Y \to Y'$ with $b_1(W) = b^+(W) = 0$, where $c \subset W$ is a closed, oriented, embedded surface.
\item If $Y$ is a homology $S^2 \times S^1$, there is an associated $\Bbb Z/4$-graded module $I_d^w(Y;R)$, again invariant under orientation-preserving diffeomorphism. 
\item We have $I_d(S^3;R) = 0$ and $I_d^w(S^2 \times S^1;R) = 0$ for all $d$. 
\end{enumerate}
\end{proposition}

\begin{proof} For integer homology spheres see \cite{Floer-ZHS}. For the case that $Y$ is a homology $S^2 \times S^1$ see \cite{Floer-knot}; this is the instanton Floer homology of $Y$ equipped with a $U(2)$-bundle with odd first Chern class. This bundle is a {\it non-trivial admissible bundle}, which supports no reducible connections. The vanishing result follows because neither $S^3$ nor $S^2 \times S^1$ (with the corresponding non-trivial $U(2)$ bundle) admits irreducible projectively flat $U(2)$-connections.
\end{proof}
\noindent More general functoriality statements hold, but we will not need them. We will also suppress the base ring $R$ from notation.

In the case of integer homology spheres, these groups have a natural enrichment to $IP$-modules.

\begin{proposition}\label{prop:IP-mod}
If $Y$ is an integer homology sphere, the graded module $I_*(Y)$ may be enriched with the structure of an $IP$-module $I_\bullet(Y)$, invariant under orientation-preserving diffeomorphism.
\end{proposition}

\begin{remark}
Keep in mind that $I_\bullet(Y)$ is the data of $F_r I_d(Y)$ for all $r \in \Bbb R \cup \{\infty\}$ and $d \in \Bbb Z$, together with connecting homomorphisms $i$ and periodicity maps $\varphi$. On the other hand, $I_*(Y)$ is the data of $I_d(Y)$ for all $d \in \Bbb Z/8$. In particular, $I_*(Y)$ is the direct sum of $F_\infty I_d(Y)$ for all $d\in \Bbb Z/8$. (Note that our assumptions on $IP$-modules imply that $F_\infty I_d(Y)$ depends only mod $8$ value of $d$.) In summary, $I_*(Y)$ is part of the information contained in $I_\bullet(Y)$.
\end{remark}

\begin{proof}
This claim is strictly weaker than the result of the main constructions of \cite{Daemi-Gamma,NST}. (The language of $IP$-modules is closer to the language used in \cite{NST}.) Both papers also pay close attention to the interaction with the reducible connection, which is irrelevant for our purposes. We briefly review the construction.  

Write $\mathcal B^*(Y)$ for the space of irreducible $SU(2)$-connections on the trivial bundle over $Y$, modulo gauge. This space is equipped with the Chern--Simons functional $\CS: \mathcal B^*(Y) \to \Bbb R/\Bbb Z$. The universal cover of this space, the space of irreducible $SU(2)$-connections modulo degree-zero gauge transformations, carries a $\Bbb Z$-periodic lift $\widetilde{\CS}: \widetilde{\mathcal B}^*(Y) \to \Bbb R$. A canonical choice of lift is determined by requiring that the natural extension to the space of \emph{all} $SU(2)$-connections modulo degree-zero gauge transformations has $\widetilde{\CS}(\theta) = 0$, where $\theta$ is the equivalence class of the trivial connection. 

So long as $r$ is not a critical value of the Chern--Simons function on $Y$, we define $F_r I_d(Y)$ to be the degree-$d$ Morse homology of the sublevel set $\widetilde{\CS}^{-1}(-\infty, r]$, with respect to an appropriately perturbed Chern--Simons functional. The degree $i(\alpha)$ of an irreducible connection $\alpha$ is defined to be the index of the (appropriately perturbed) ASD operator $D_A^+$ associated to a connection on $\Bbb R \times Y$ which is equal to $\alpha$ at $-\infty$ and the trivial connection $\theta$ at $+\infty$. 

For $r$ not a critical value, that this Morse homology is well-defined is given as \cite[Lemma 2.6]{NST} (when comparing, take $s = -\infty$). To simplify the definition of the $IP$-module, we extend the definition of $F_r I_d(Y)$ to include the critical values by demanding this assignment be right-continuous: for $r$ a critical value we set $F_r I_d(Y) = F_{r+\epsilon} I_d(Y)$ for sufficiently small $\epsilon > 0$.
\end{proof}

The functoriality of Proposition \ref{prop:Floer-def} extends to $IP$-morphisms.

\begin{proposition}\label{prop:IP-morphism}Suppose $W: Y \to Y'$ is a cobordism between integer homology spheres with $b_1(W) = b^+(W) = 0$, and $c \subset W$ is an embedded oriented surface. Then there is a constant $\eta(W, c) \ge 0$ which is strictly positive if $\pi_1(W) = 0$ and an induced $IP$-morphism $(W,c)_*: I_\bullet(Y) \to I_\bullet(Y')$ of degree $D = -2c^2$ and level $L = -c^2/4 - \eta(W,c)$, which enriches the cobordism maps of Proposition \ref{prop:Floer-def}(a). 
\end{proposition}

\begin{remark}
The data of $(W,c)_*: I_\bullet(Y) \to I_\bullet(Y')$ consists of induced maps $(W,c)_*: F_r I_d(Y) \to F_{r+L} I_{d+D}(Y')$ for all $r \in \Bbb R \cup \{\infty\}$ and $d \in \Bbb Z$ compatible with the connecting homomorphisms $i$ and periodicity maps $\varphi$. In particular, we use the same notation $(W,c)_*$ for the induced map $I_d(Y) \to I_{d+D}(Y')$ on Floer's instanton homology.
\end{remark}

\begin{proof}
When $c = \varnothing$, the final claim is established for a more complicated chain complex in \cite[Proposition 2.15, Lemma 2.35]{Daemi-Gamma} and similarly \cite[Lemma 2.10-2.11]{NST}. These maps arise by counting irreducible connections of index zero on the trivial $SU(2)$-bundle over $W$ satisfying an appropriately perturbed ASD equation. While the latter reference only defines these maps for $r$ outside the critical set of the two Chern--Simons functions, again we may extend them to all $r$ in a right-continuous fashion.

The cobordism map for nonempty $c$ is given by the same construction, counting connections on the $U(2)$-bundle $E_c$ over $W$ with $c_1(E_c)$ Poincar\'e dual to $c$, for which the traceless part of the curvature $F_0(A)$ satisfies a perturbed version of the ASD equation $F_0(A)^+ = 0$. Here we count connections whose induced connection on the determinant bundle is fixed, modulo determinant-$1$ gauge transformations. The proof that this map is well-defined goes through with no change, but we will discuss the computation of its degree and level.

Given critical values $\alpha, \alpha'$ of the perturbed Chern--Simons functional of $Y$, $Y'$, we can form the moduli space of connections on $E_c$ satisfying the perturbed ASD equation and asymptotic to $\alpha$ and $\alpha'$ on the ends. This moduli space has connected components with different expected dimensions. Furthermore, the expected dimension of the connected component containing a connection $A$ is uniquely determined by the \emph{topological energy} of $A$, defined as 
\begin{equation}\label{eqn:energy}
\mathcal E(A) = \frac{1}{8\pi^2} \int \text{tr}(F_0(A)^2).
\end{equation}
We use the following convention to fix a subspace $M(W, c; \alpha, \alpha')$  of the moduli space of instantons on $E_c$ with a fixed expected dimension.

The bundle $E_c$ has a splitting as $L_c\oplus \underline {\mathbb C}$ where $L_c$ is a line bundle whose $c_1$ is Poincar\'e dual to $c$. If $A_c$ is a connection on $L_c$ which is asymptotic to the trivial connection on the ends, then it induces a reducible connection on $E_c$ compatible with the splitting $L_c\oplus \underline {\mathbb C}$, which we still denote by $A_c$. The quantity $\mathcal E(A_c)$ is equal to $-c^2/4$ by Chern--Weil theory. Then the moduli space $M(W, c; \alpha, \alpha')$ consists of gauge equivalence classes of instantons $A$ on $E_c$ that are asymptotic to $\alpha, \alpha'$ on the ends and whose topological energy satisfies 
\begin{align}\label{eqn:energy-rel}
\mathcal E(A) &= \mathcal E(A_c)+ \widetilde{\CS}(\alpha) - \widetilde{\CS}(\alpha')\nonumber\\&= -c^2/4+\widetilde{\CS}(\alpha) - \widetilde{\CS}(\alpha').
\end{align}

Now we turn to the definition of $(W, c)_*$. The coefficient of $\alpha'$ in $(W, c)_* (\alpha)$ is given by counting the number of points in $M(W, c; \alpha, \alpha')$ when this space has expected dimension zero. Observe that by additivity of the ASD index \cite[Equation (3.2), Proposition 3.10]{DonBook}, we have 
\begin{align}\label{index-form}i(W,c; \alpha, \alpha') 
&= i(\alpha) + 3 + i(W, c; \theta, \theta') - i(\alpha').\end{align} 
Here $i(W,c; \alpha, \alpha')$ denotes the index of the ASD operator associated to connections in $M(W, c; \alpha, \alpha')$, and $i(W, c; \theta, \theta')$ denotes the index of $A_c$.  
Using the index formula for the ASD operators, we have $i(W, c; \theta, \theta') = -2c^2 - 3.$ Using this and \eqref{index-form}, $i(W,c; \alpha, \alpha') = 0$ if and only if $i(\alpha') = i(\alpha) - 2c^2$, giving the degree.

To determine the level, we argue similarly. To simplify the discussion, we assume that no perturbation is needed to define the cobordism map $(W,c)_*$, and we refer the reader to \cite[Lemma 3.48]{Daemi-Gamma} for the limiting argument used to address the more general case; see also the proof of \cite[Theorem 3.7]{NST} for a similar limiting argument.  Following \cite[Definition 3.46]{Daemi-Gamma}, let $\eta(W,c)$ denote the least topological energy of any (possibly broken) ASD connection on $(W,E_c)$ with irreducible flat limits. Since the topological energy of any instanton is non-negative, the quantity $\eta(W,c)$ is also non-negative. Furthermore, if $\eta(W,c)$ is zero, then $E_c$ admits some projectively flat $U(2)$-connection with irreducible limits, and therefore its adjoint bundle admits a flat connection with irreducible limits. It follows that there is a homomorphism $\pi_1(W) \to SO(3)$ which restricts to a non-trivial homomorphism on both ends. It follows that if $\pi_1(W) = 0$, then $\eta(W, c) > 0$. 
In the definition of $\langle (W,c)_*\alpha, \alpha'\rangle$, we count instantons with irreducible flat limits $\alpha$ and $\alpha'$. Therefore, \eqref{eqn:energy-rel} implies that 
 $$\widetilde{\CS}(\alpha') = \widetilde{\CS}(\alpha) -c^2/4 - \mathcal E(A) \le \widetilde{\CS}(\alpha) - c^2/4 - \eta(W,c).$$
This gives the claim about the level of $(W,c)_*$. 
\end{proof}

Proposition \ref{prop:IP-mod} allows us to define numerical invariants of integer homology spheres, while Proposition \ref{prop:IP-morphism} --- together with Lemma \ref{lemma:kappa-ineq} --- allows us to analyze their behavior under cobordisms which induce injections on Floer's instanton homology. We will now focus on $R = \Bbb F_2$.

\begin{definition}
Suppose $Y$ is an integer homology sphere. We define $\kappa_Y: \Bbb Z \to \Bbb R \cup \{\infty\}$ and $\ell(Y) \in \Bbb R \cup \{\infty\}$ as $$\kappa_Y(d) = \kappa_{I_\bullet(Y; \Bbb F_2)}(d), \quad \ell(Y) = \ell(I_\bullet(Y; \Bbb F_2)),$$ the numerical invariants of the $IP$-module associated with $Y$. 
\end{definition}

These quantities are invariant under orientation-preserving diffeomorphisms of integer homology spheres. Notice that Lemma \ref{lemma:kappa-basic}(c) immediately implies that $\ell(Y) < \infty$ if and only if $I_d(Y; \Bbb F_2) \ne 0$ for some $d \in \Bbb Z/8$.

\begin{example}\label{ex:ell}
We will compute $\ell$ for the Poincar\'e homology sphere $\Sigma(2,3,5)$ oriented as the boundary of the negative definite $E8$ plumbing, or equivalently, as $-1$-surgery on the left-handed trefoil. As discussed in \cite[Lemma 2.9(i)]{DLME}, there are two flat, irreducible $SU(2)$-connections modulo gauge, denoted $[\alpha]$ and $[\beta]$, with $\CS(\alpha) = 1/120$ and $\CS(\beta) = 49/120$, with $\alpha$ in grading $1$. The computations there extend to show that $\beta$ has grading $5$. It follows that 
\[\kappa_{\Sigma(2,3,5)}(d) = 
\begin{cases} 
1/120 + (d-1)/8 & d \equiv 1 \mod 8 \\
49/120 + (d-5)/8 & d \equiv 5 \mod 8 \\ 
\infty & d \not\equiv 1 \mod 4,
\end{cases}\]
and thus that

\[\ell(\Sigma(2,3,5)) = \text{min}(1/120 - 1/8, 49/120 - 5/8) = \text{min}(-7/60, -13/60) = -13/60.\]
\end{example}

\subsection{Exact triangles}
We state the two exact triangles relevant to this paper. The first is now classical, and due to Floer. The second originates, in the context of admissible bundles, in \cite{CDX}. In this section, $Y$ is an integer homology sphere, $K \subset Y$ is a knot, and $Y_r(K)$ is the manifold obtained by $r$-surgery on $K$.

The following is \cite[Theorem 2.4]{Floer-knot}. Our index convention compares to Floer's as $i(\alpha) = -3 - i_{\text{Floer}}(\alpha)$. Here, we collapse the $\Bbb Z/8$-grading on $I_*(Y)$ to a $\Bbb Z/4$-grading, so for instance the degree $1$ mod $4$ part of $I_*(Y)$ is $I_1(Y) \oplus I_5(Y)$. 

\begin{proposition}\label{prop:Floer-triangle}
There is an exact triangle of $\Bbb Z/4$-graded modules \[\begin{tikzcd}
	{I_{*}(Y)} && {I_*(Y_{-1}(K))} \\
	\\
	& {I_*^w(Y_0(K))}
	\arrow["{[-3]}", from=3-2, to=1-1]
	\arrow["{W_*}", from=1-1, to=1-3]
	\arrow[from=1-3, to=3-2]
\end{tikzcd}\]
The horizontal homomorphism is the cobordism map associated to the $2$-handle cobordism, and the leftmost homomorphism has degree $-3$.
\end{proposition}

In particular, by Proposition \ref{prop:IP-morphism}, the homomorphism $W_*: I_*(Y) \to I_*(Y_{-1})$ naturally extends to an $IP$-morphism $I_\bullet(Y) \to I_\bullet(Y_{-1})$ of degree and level zero.

We will establish the following distance-two surgery exact triangle only over $R = \Bbb F_2$, to avoid checking tedious details with signs. The result is expected to hold with coefficients in any commutative ring.

\begin{proposition}\label{prop:CDX-triangle}
There is an exact triangle of $\Bbb Z/8$-graded $\Bbb F_2$-vector spaces
\[\begin{tikzcd}
	{I_{*+2}(Y_{-1}(K); \Bbb F_2)} && {I_*(Y_1(K); \Bbb F_2)} \\
	\\
	& {I_*(Y; \Bbb F_2) \oplus I_{*+2}(Y; \Bbb F_2)}
	\arrow["{(W',c')_* \oplus W'_*}", from=3-2, to=1-1]
	\arrow["{[-1]}", from=1-1, to=1-3]
	\arrow["{W_* \oplus (W, c)_*}", from=1-3, to=3-2]
\end{tikzcd}\]
The diagonal homomorphisms are induced by direct sums of cobordism maps, where $W$ and $W'$ are the 2-handle cobordisms with $c, c'$ the cocore and core respectively, capped off with a Seifert surface in $Y$ to give a closed surface.
\end{proposition}

Here, because of the degree shift in the domain, the statement that the top homomorphism $I_{*+2}(Y_{-1}) \to I_*(Y_1)$ has degree $-1$ means that $I_d(Y_{-1})$ maps to $I_{d-3}(Y_1)$. 

Now Proposition \ref{prop:IP-morphism} states that the rightmost and leftmost homomorphisms have natural extensions to a \emph{direct sum} of $IP$-morphisms. The morphisms $W_*$ and $W'_*$ have degree zero and nonpositive level, while $(W,c)_*$ and $(W',c')_*$ have degree $2$ and level at most $1/4$. It will be relevant later that $L - D/8 \le 0$ in both cases, with strict inequality if $\pi_1(Y \setminus K)$ is normally generated by the meridian $\mu_K$. The proof of Proposition \ref{prop:CDX-triangle} will be given in Section \ref{sec:proof-36}.\\

When $Y = S^3$, the top map is an isomorphism because $I_*(S^3)$ is trivial. The proof of the exact triangle in this particular case gives an explicit description of its inverse, a map $g_1: I_*(Y_1(K); \Bbb F_2) \to I_{*-1}(Y_{-1}(K); \Bbb F_2)$. This map is obtained by counting instantons on the composite cobordism with respect to a certain \emph{one-parameter family} of metrics on the composite. As a result, it gives rise to an $IP$-morphism. We record its properties in the following proposition, whose proof is given in Section \ref{sec:proof-37}.

\begin{proposition}\label{prop:pm-one-map}
Let $K$ be a knot in $S^3$. Then there is an $IP$-morphism $g_1: I_\bullet(S^3_1(K);\Bbb F_2) \to I_\bullet(S^3_{-1}(K); \Bbb F_2)$ of degree $3$ and level $1/4 - \eta(K)$, where $\eta(K) \ge 0$ is strictly positive when $K$ is not the unknot. The induced map on Floer's instanton homology $I_d(S^3_1(K); \Bbb F_2) \to I_{d-1}(S^3_{-1}(K); \Bbb F_2)$ is an isomorphism.
\end{proposition}

\section{Proof of the main theorems}\label{sec:cosmetic-proof}

Here we use the results stated in the previous section to prove our main results.

\subsection{Proof of Theorem \ref{thm:S2timesS1}}
Theorem \ref{thm:S2timesS1} follows from the following more precise claim. To set notation, let $L$ be a knot in $S^2 \times S^1$ which generates the ihomology group  $H_1(S^2 \times S^1;\Z)$ and fix a choice of framing curve $\lambda$. Let $S_n(L)$ be the three-manifold obtained by surgery on $L$ with framing $n\mu + \lambda$.  Note that $S_n(L)$ is a homology sphere for every $n$.  We write $K_n \subset S_n(L)$ for the dual knot.

\begin{theorem}\label{ineq-S^1xS^2}
In the situation above, if $L$ is not isotopic to $\{\ast\} \times S^1$, then we have $$-\infty < \ldots < \ell(S_{n+1}(L)) < \ell(S_n(L)) < \ell(S_{n-1}(L)) < \ldots < \infty.$$
\end{theorem}

\begin{proof}
Our assumption on $L$ implies that $S_n(L) \setminus N(K_n) \cong S^2 \times S^1 \setminus N(L)$ is an irreducible 3-manifold whose boundary is incompressible. By \cite[Theorem 1.3]{LPCZ}, each $S_n(L)$ has non-trivial instanton homology, hence $\ell(S_n(L)) < \infty$ for all $n$.

The 2-handle cobordism $W(n): S_n(L) \to S_{n+1}(L)$ has $b^+= 0$ and trivial $\pi_1$.  To see the latter claim, note that the Seifert--Van Kampen theorem implies that $\pi_1(W(n))$ is isomorphic to the quotient of $\pi_1(S^2 \times S^1 \setminus N(L))$ by the normalizer of $\pi_1(\partial N(L))$. This is equivalent to the quotient of $\pi_1(S^2 \times S^1)$ by the normal subgroup generated by the class of $L$, and hence it is trivial. Proposition \ref{prop:IP-morphism} gives an $IP$-morphism $W(n)_*: I_\bullet(S^3_n(L)) \to I_\bullet(S^3_{n+1}(L))$ with degree $0$ and level $-\eta\big(W(n)\big) < 0$. Applying Proposition \ref{prop:Floer-triangle} to the pair $(S_n(L), K_n)$ gives an exact triangle 
\[\begin{tikzcd}
	{I_{*}(S_n(L))} && {I_*(S_{n+1}(L))} \\
	\\
	& {I_*^w(S^2 \times S^1)}
	\arrow["{[-3]}", from=3-2, to=1-1]
	\arrow["{W(n)_*}", from=1-1, to=1-3]
	\arrow[from=1-3, to=3-2]
\end{tikzcd}\]

By Proposition \ref{prop:Floer-def}(c), $I^w_*(S^2 \times S^1) = 0$, so the induced map $W(n)_*: I_*(S_n(L)) \to I_*(S_{n+1}(L))$ is an isomorphism. By Lemma \ref{lemma:kappa-ineq}(b), we see that $$\ell(S_{n+1}(L)) \le \ell(S_n(L)) - \eta(W(n)).$$ Because $\ell(S_n(L))$ is finite, we in fact have a strict inequality $\ell(S_{n+1}(L)) < \ell(S_n(L))$ for all $n$.
\end{proof}

\begin{proof}[Proof of Theorem \ref{thm:S2timesS1}]
	Theorem \ref{ineq-S^1xS^2} implies that for any pair of distinct integers $n$ and $m$, the 3-manifolds $S^3_{n}(L)$ and 
	$S^3_{m}(L)$ are not orientation-preserving diffeomorphic. Next, let $X:S^3_{n}(L) \to S^3_{m}(L)$ be a ribbon homology cobordism, and 
	let $\overline X:S^3_{m}(L) \to S^3_{n}(L)$ be the cobordism obtained by flipping around the orientation-reversal of $X$.  Proposition \ref{prop:IP-morphism} implies that
	\[
	  X_*: I_\bullet(S^3_{n}(L)) \to I_\bullet(S^3_{m}(L)),\hspace{1cm} \overline X_*: I_\bullet(S^3_{m}(L)) \to I_\bullet(S^3_{n}(L))
	\]
	are $IP$-morphisms of degree $0$ and level $0$. Furthermore, \cite[Theorem 4.1]{ribbon-hom} asserts that 
	$X_*:I_*(S_n(L)) \to I_*(S_{m}(L))$ and $\overline X_*:I_*(S_{m}(L)) \to I_*(S_{n}(L))$ are respectively 
	injective and surjective homomorphisms of vector spaces in any degree. Since it is shown in the proof of Theorem \ref{ineq-S^1xS^2} that
	$I_*(S_n(L))$ and $I_*(S_{m}(L))$ are isomorphic, the maps $X_*$ and $\overline X_*$ are also isomorphisms on Floer homology, which extend to IP-morphisms of degree $0$ and level $0$. 
	By another application of Lemma \ref{lemma:kappa-ineq}(b), we conclude that $ \ell(S_n(L)) = \ell(S_{m}(L))$, which is a contradiction.
\end{proof}

\subsection{Proof of Theorem \ref{thm:1/n}}
The proof of Theorem \ref{thm:1/n} is similar, but more intricate, and requires some more initial input. Henceforth, $K$ denotes a knot in $S^3$. We begin with a computation of instanton Floer groups for surgeries on a knot. 

The following statement is restricted to the case $R = \Bbb F_2$ because Proposition \ref{prop:CDX-triangle} is. An interested reader willing to lift the proof of Proposition \ref{prop:CDX-triangle} to the integers would be able to generalize the following statement to allow coefficients in an arbitrary commutative ring.

\begin{lemma}\label{lemma:n-surgery-calc}
For any integer $n > 0$ we have an isomorphism of $\Bbb Z/8$-graded $\Bbb F_2$-vector spaces $$I_*\big(S^3_{-1/n}(K); \Bbb F_2\big) \cong \bigoplus_{i=0}^{n-1} I_{*-2i}\big(S^3_{-1}(K); \Bbb F_2\big), \quad \quad I_*\big(S^3_{1/n}(K);\Bbb F_2\big) \cong \bigoplus_{i=0}^{n-1} I_{*+2i}\big(S^3_{1}(K); \Bbb F_2\big).$$ Furthermore, we have isomorphisms of $\Bbb Z/4$-graded $\Bbb F_2$-vector spaces $$I_*(S^3_{-1}(K); \Bbb F_2) \cong I_*^w(S^3_0(K); \Bbb F_2) \cong I_{*-3}(S^3_1(K); \Bbb F_2).$$
\end{lemma}
\begin{proof}
The coefficient ring $R = \Bbb F_2$ will be suppressed from notation for the rest of the argument. We will prove the first statement for $I_*(S^3_{1/n}(K))$ for $n > 0$; the argument for $I_*(S^3_{-1/n}(K))$ can be proved similarly. To simplify notation, for this proof we write $S_{1/n} = S^3_{1/n}(K)$. 

We will prove the following stronger claim by induction on $n \ge 1$: \begin{itemize}
\item For each $n$ there is an isomorphism $$\varphi_n: I_*(S_{1/n}) \to \bigoplus_{i=0}^{n-1} I_{*+2i}(S_1).$$
\item These isomorphisms can be chosen so that for all $n \ge 2$, the 2-handle cobordism map $I_*(S_{1/n}) \to I_*(S_{1/(n-1)})$ is identified with the projection to the first $n-1$ coordinates (and in particular is surjective).\footnote{The corresponding statement for $-1/n$ is that the 2-handle cobordism map $W_*: I_*(S_{-1/(n-1)}) \to I_*(S_{-1/n})$ is identified with the inclusion of the first $n-1$ coordinates, and in particular is injective.}
\end{itemize}

The base case $n = 1$ is tautological; one may take $\varphi_1$ to be the identity.

Suppose the claim is proved for $n \ge 1$. Apply Proposition \ref{prop:CDX-triangle} to the pair $(Y, K) = (S_{1/n}, \tilde K)$ with $\tilde K$ being the dual knot. By inductive hypothesis the 2-handle cobordism map $W'_*: I_*(S_{1/n}) \to I_*(S_{1/(n-1)})$ is known to be surjective, so our exact triangle is in fact a short exact sequence $$0 \to I_*(S_{1/(n+1)}) \xrightarrow{W_* \oplus (W,c)_*} I_*(S_{1/n}) \oplus I_{*+2}(S_{1/n}) \xrightarrow{(W', c')_* \oplus W'_*} I_{*+2}(S_{1/(n-1)}) \to 0.$$ 

Applying the inductive hypothesis, this sequence is isomorphic to the short exact sequence $$0 \to  I_*(S_{1/(n+1)}) \xrightarrow{\varphi_n W_* \oplus \varphi_n (W,c)_*} \left(\bigoplus_{i=0}^{n-1} I_{*+2i}(S_1)\right) \oplus \left(\bigoplus_{j=1}^n I_{*+2j}(S_1)\right) \xrightarrow{f \oplus \pi} \bigoplus_{j=1}^{n-1} I_{*+2j}(S_1) \to 0,$$ where $\pi$ is projection onto the first $n-1$ coordinates. 

Thus by exactness the map $\varphi_n W_* \oplus \varphi_n (W,c)_*$ induces an isomorphism \begin{align*}I_*(S_{1/(n+1)}) &\cong \{(x,y,z) \in \left(\bigoplus_{i=0}^{n-1} I_{*+2i}(S_1)\right) \oplus \left(\bigoplus_{j=1}^{n-1} I_{*+2j}(S_1)\right) \oplus I_{*+2n}(S_1)  \mid f(x) + y=0\} \\
&\cong \left(\bigoplus_{i=0}^{n-1} I_{*+2i}(S_1)\right) \oplus I_{*+2n}(S_1),
\end{align*} where the final map sends $(x,y,z)$ to $(x,z)$. The isomorphism $\varphi_{n+1}$ is the composite of these two identifications. Finally, with respect to these isomorphisms the map $I_*(S_{1/(n+1)}) \to I_*(S_{1/n})$ is given by sending $(x,z)$ to $x$, completing the induction.\\

The second claim is Proposition \ref{prop:Floer-triangle} applied to the pairs $(Y,K) = (S^3, K)$ and $(S^3_1(K), \tilde K)$.
\end{proof}

\begin{corollary}\label{cor:nonvanishing}
If $K$ is a non-trivial knot in $S^3$, the vector space $I_*(S_{1/n}(K); \Bbb F_2)$ is non-trivial for all integers $n \ne 0$.
\end{corollary}
\begin{proof} 
By the preceding lemma, it is equivalent to show $I^w_*(S_0(K); \Bbb F_2) \ne 0$. By the universal coefficient theorem, it suffices to show $I^w_*(S_0(K); \Bbb C) \ne 0$. Finally, because $K$ is non-trivial, \cite[Proposition 7.22]{KM-sutures} asserts that $I^w_*(S_0(K); \Bbb C)$ is non-trivial.
\end{proof}

Finally, Theorem \ref{thm:1/n} follows immediately from the following more precise claim.

\begin{theorem}
If $K$ is a non-trivial knot in $S^3$, we have $$-\infty < \cdots < \ell\big(S^3_{-1/2}(K)\big) < \ell\big(S^3_{-1}(K)\big) < \ell\big(S^3_{1}(K)\big) < \ell\big(S^3_{1/2}(K)\big) < \cdots < \infty$$

In fact, we have $\ell\big(S^3_{-1/n}(K)\big) < \ell\big(S^3_{1/n}(K)\big) - 1/8$ for all $n > 0$.
\end{theorem}
\begin{proof}
That these $\ell$-invariants are all finite follows from Corollary \ref{cor:nonvanishing} and Lemma \ref{lemma:kappa-basic}(c).

It was established in the proof of Lemma \ref{lemma:n-surgery-calc} that for $n > 0$, the cobordism $W_n: S^3_{-1/n}(K)  \to S^3_{-1/(n+1)}(K)$ induces an injection on Floer homology. Thus, Floer's exact triangle collapses to a short exact sequence $$0 \to I_*(S^3_{-1/n}(K)) \to I_*(S^3_{-1/(n+1)}(K)) \to I_*^w(S^3_0(K)) \to 0,$$ where the first map $(W_n)_*$ is the cobordism map induced by the simply-connected negative-definite cobordism $W_n: S^3_{-1/n}(K) \to S^3_{-1/(n+1)}(K)$ given by attaching a $(-1)$-framed handle along $\tilde K$. Because $(W_n)_*$ is injective, and by Proposition \ref{prop:IP-morphism} extends to an $IP$-morphism $$I_\bullet\big(S^3_{-1/n}(K)\big) \to I_\bullet\big(S^3_{-1/(n+1)}(K)\big)$$ of degree $0$ and level $-\eta(W_n) < 0$, it follows from Lemma \ref{lemma:kappa-ineq}(b) that $\ell(S^3_{-1/(n+1)}(K)) < \ell(S^3_{-1/n}(K))$ for all $n > 0$. 

Next, Proposition \ref{prop:pm-one-map} and Lemma \ref{lemma:kappa-ineq}(b) immediately combine to give 
$$\ell\big(S^3_{-1}(K)\big) \le \ell\big(S^3_{1}(K)\big) + \big(1/4 - \eta(K) - 3/8\big) = \ell\big(S^3_{1}(K)\big) - 1/8 - \eta(K) < \ell\big(S^3_{1}(K)\big) - 1/8.$$
Finally, we use that the triangle of Proposition \ref{prop:CDX-triangle} also collapses into a short exact sequence $$0 \to I_*(S^3_{1/(n+1)}(K)) \to I_*(S^3_{1/n}(K)) \oplus I_{*+2}(S^3_{1/n}(K)) \to I_{*+2}(S^3_{1/(n-1)}(K)) \to 0.$$ In particular, the first map is injective. This map is the direct sum $W_* \oplus (W,c)_*$ of cobordism maps, where $W$ is simply-connected and negative-definite. Proposition \ref{prop:IP-morphism} implies that this enriches to a direct sum of morphisms of $IP$-modules, the first of which has $L_1 - D_1/8 = -\eta(W) < 0$, the latter of which has $L_2 - D_2/8 = -\eta(W,c) < 0$. In particular, the larger of the two is still negative. Lemma \ref{lemma:inhomog-ineq} immediately gives $$\ell\big(S^3_{1/n}(K)\big) < \ell\big(S^3_{1/(n+1)}(K)\big),$$ completing the proof.
\end{proof}

\begin{proof}[Proof of Theorem \ref{thm:1/n}]
By Lemma \ref{lemma:n-surgery-calc}, we see that $I_*(S^3_{1/n}(K); \Bbb F_2)$ and $I_*(S^3_{-1/n}(K); \Bbb F_2)$ have the same rank. As in the proof of Theorem \ref{thm:S2timesS1}, we see that if $X: S^3_{1/n}(K) \to S^3_{-1/n}(K)$ is a ribbon homology cobordism and $\overline X$ is obtained by flipping its orientation-reversal, then the induced maps $X_*, \overline X_*$ are both isomorphisms. From this it follows that $\ell(S^3_{1/n}(K)) = \ell(S^3_{-1/n}(K))$, a contradiction.
\end{proof}

\section{The distance-two surgery triangle}\label{sec:tri-proof}

The goal of this section is to construct the exact triangle in Proposition \ref{prop:CDX-triangle}. The existence of such an exact triangle is proposed in \cite{CDX}, and our proof here follows closely the proof of \cite[Theorem 1.6]{CDX}, which concerns the analogue of Proposition \ref{prop:CDX-triangle} for instanton homology of admissible bundles with respect to the more general gauge group $SU(N)$. In particular, the proof of the $N=2$ case of \cite[Theorem 1.6]{CDX} provides the skeleton of the proof of Proposition \ref{prop:CDX-triangle}, except that we also need to analyze the reducible ASD connections over various cobordisms to guarantee that they do not cause any issue in the construction of the distance-two surgery triangle. Our proof of Proposition \ref{prop:CDX-triangle} is also formally similar to the proof of Floer's surgery exact triangle as given in \cite{Scaduto-KH}, though the argument is complicated by the presence of the more complicated `middle ends' discussed in Section \ref{sec:cobordism-defs}.

In the first subsection below, we review a standard homological algebra lemma about exact triangles. In the next subsection, we review the definition of various cobordisms with families of Riemannian metrics relevant in the proof of Proposition \ref{prop:CDX-triangle}. Then we study the reducible ASD connections with respect to these families of metrics. The proofs of Proposition \ref{prop:CDX-triangle} and Proposition \ref{prop:pm-one-map} are given in the final two subsections. 

To avoid dealing with the study of orientations of moduli spaces, we work with $\Bbb F_2$ coefficients in this section. In particular, all chain complexes in this section are defined over $\Bbb F_2$.

\subsection{Homological algebra of surgery exact triangles}\label{sec:triangle-detection}
As with many exact triangles in Floer theory, the proof of Proposition \ref{prop:CDX-triangle} is given by the \emph{triangle detection lemma} (see \cite[Lemma 3.7]{seidel} and \cite[Lemma 4.2]{os:branched}).
\begin{proposition}\label{T-detec}
	For each $i \in \Z$, let $(C_i,d_i)$ be a chain complex.
	Suppose that for all $i\in \Z$ we are given maps
	\[
	  f_i:C_i \to C_{i-1} \hspace{.5cm} g_i: C_i \to C_{i-2} \hspace{.5cm} h_i: C_i \to C_{i-3},
	\]
	which satisfy the following properties:
	\begin{eqnarray*}
		d_i^2&=& 0 \label{identity-0}\\
		d_{i-1}  f_i + f_i  d_i  &=& 0 \label{identity-1}\\
		d_{i-2} g_i + f_{i-1}  f_i + g_i  d_i &=& 0 \label{identity-2} \\
		d_{i-3} h_i + f_{i-2}  g_i + g_{i-1}  f_i + h_i  d_i &=& q_i \label{identity-3}
	 \end{eqnarray*}
	where $q_i:C_i \to C_{i-3}$ is homotopic to the identity.
	Then the map
	\[
	  C_{i+1} \xrightarrow{(f_i,g_i)} \textup{Cone}(f_{i-1}) \overset{\textup{def}}{=} \left(C_{i-1}\oplus C_{i-2}[1],\left[\begin{array}{cc} d_{i-1}&0\\ f_{i-1} & d_{i-2}\end{array}\right]\right)  \\
	\]
	is a chain homotopy equivalence. In particular, if $H_i$ denotes the homology of $(C_i,d_i)$, then we have the following exact triangle:
	\[
	\begin{tikzcd}
	H_{-1} && H_{1} \\
	\\
	&H_{0}
	\arrow["{(f_{0})_*}", from=3-2, to=1-1]
	\arrow["{(f_{-1})_*}", from=1-1, to=1-3]
	\arrow["{(f_1)_*}", from=1-3, to=3-2]
\end{tikzcd}
	\]
\end{proposition}
\begin{remark}
The final hypothesis on $q_i$ is equivalent to the statement that $f_{i-2} g_i + g_{i-1} f_i$ is homotopic to the identity, so this condition could be stated without reference to $h_i$ or $q_i$. However, when we apply the triangle detection lemma, we will establish that this map is homotopic to the identity by defining a homotopy to a map which is then seen to be homotopic to the identity. We thus state the lemma in the form most closely resembling the way in which it is applied below. 
\end{remark}

To prove Proposition \ref{prop:CDX-triangle}, we apply this lemma to the case that 
\[
  C_{-1} := CI_*(Y_{-1}(K)),\hspace{1cm} C_{0}:=CI_*(Y) \oplus CI_{*-2}(Y),\hspace{1cm}C_{1}:=CI_*(Y_1(K)).
\]
More generally, $C_k$ is defined by requiring that $C_k=C_{k+3}$ for any $k$. The homomorphism $d_i$ in each case is given by the corresponding Floer differential and the maps $f_i$, $g_i$ and $h_i$ are given by cobordism maps in instanton Floer theory. We recall the definition of these cobordisms in the next subsection. 

\subsection{Cobordisms and families of metrics}\label{sec:cobordism-defs}

First we fix some notations for the discussion of the cobordisms involved in the proof of Proposition \ref{prop:CDX-triangle}. A cobordism from a 3-manifold $Z$ to another 3-manifold $Z'$ with a middle end $L$ is a 4-manifold $W$ with
\[
\partial W=-Z\sqcup Z'\sqcup L.
\] 
We write $W:Z \xrightarrow{L}Z'$ for any such cobordism, and we drop $L$ from the notation when the choice of the middle end is clear from the context. Given two such cobordisms $W_0:Z \xrightarrow{L}Z'$ and $W_1:Z' \xrightarrow{L'}Z''$, we may compose them to obtain $W_0\circ W_1:Z\xrightarrow{L\sqcup L'}Z''$. A metric with cylindrical ends on $W$ is a Riemannian metric on the non-compact 4-manifold given by gluing to $W$ the cylinders $(-\infty,0]\times Z$, $[0,\infty)\times Z'$ and $[0,\infty)\times L$ along the boundary components of $W$ such that the metric on the cylinders are the product metric. We usually assume that Riemannian metrics on the boundary components of $W$ are already fixed, and the product metrics on the cylindrical ends are modeled on these metrics.

In the following, fix a knot $K$ in an integer homology sphere $Y$, and let $E(K)$ denote the exterior of $K$. We assume that an identification of $\partial E(K)$ with $S^1\times S^1$ is fixed. For $-1\leq i\leq 1$, let $Z_i$ denote the result of $1/i$ surgery on the knot $K$ in $Y$, and extend the definition of $Z_i$ to any integer $i$ by requiring that $Z_{k+3}=Z_k$ for any $k$.

We first construct a cobordism $X: Z_{-1} \to Z_1$ with middle end $\mathbb{RP}^3$. A schematic diagram of this construction is presented as Figure \ref{fig1} below. Form a 4-manifold by gluing $[-3,-1]\times S^1\times D^2$ and $[1,3]\times S^1\times D^2$ to $[-3,3]\times E(K)$ respectively along $[-3,-1]\times \partial E(K)$ and $[1,3]\times \partial E(K)$ using the identifications ${\bf 1}_{[-3,-1]} \times \varphi_{-1}$ and ${\bf 1}_{[1,3]} \times \varphi_{1}$ where $\varphi_{\pm 1}:\partial (S^1\times D^2) \to \partial E(K)$  corresponds to $\pm 1$-surgery on $K$. The resulting $4$-manifold $X$ has three boundary components (after smoothing the corners) which may be identified with $-Z_{-1}$, $Z_1$ and $\mathbb{RP}^3$. This is a special case of the construction described in \cite[Section 3.1]{CDX}.

Using a similar construction, we construct cobordisms 
\[
  X':Z_1\xrightarrow{S^3} Z_0,\hspace{1cm} X'':Z_0\xrightarrow{S^3} Z_{-1}.
\]
Filling the $S^3$ boundary components of the latter two cobordisms with $4$-balls give rise to the standard 2-handle cobordisms. For any integer $i$, let $W^{i}_{i-1}:Z_i \xrightarrow{L_i} Z_{i-1}$ be given by either of $X$, $X'$ or $X''$ where $L_i$ is either $\mathbb{RP}^3$ or $S^3$. More generally, we define $W^{i}_{j}:Z_i \xrightarrow{L} Z_{j}$ for any $j\leq i$ as the iterated composite
\[
  W^{i}_{j}:= W^{j+1}_{j}\circ \cdots \circ W^{i}_{i-1}
\]
with $L=L_{i}\sqcup \dots \sqcup L_{j+1}$.

\begin{figure}
  \centering
\includegraphics[width=4.5in]{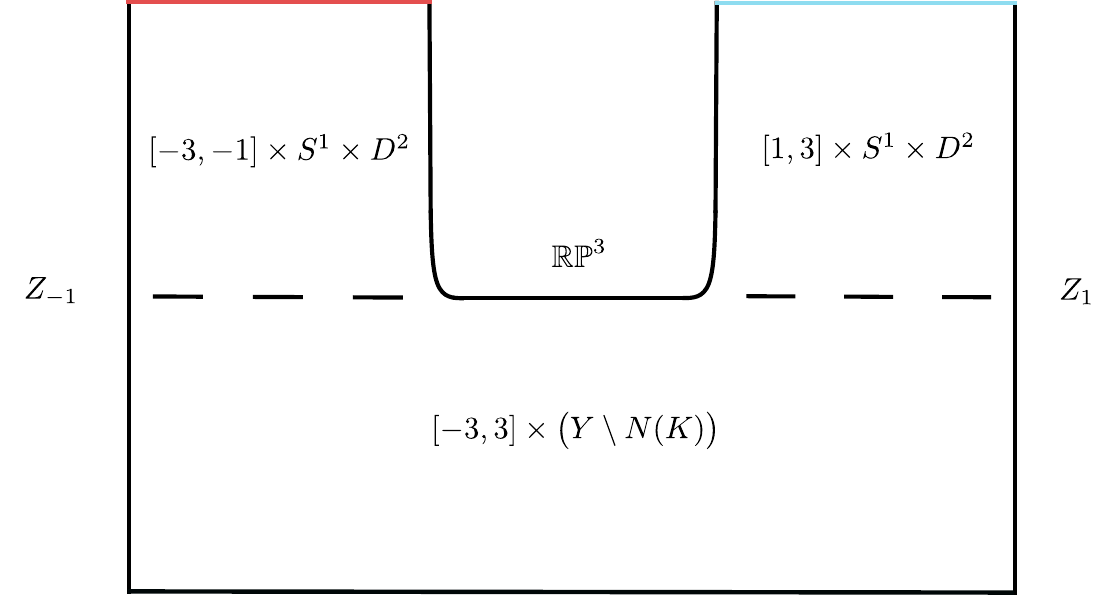}
\caption{A representation of the cobordism $X = W^{-1}_{-2}$. The three boundary components are labeled by their oriented diffeomorphism types. The embedded surface $\mathfrak c_- = [-3,-1] \times S^1 \times \{0\}$ is represented as the darker red curve, while $\mathfrak c_+$ is represented as the lighter blue curve. We visualize $c_+$ and $c_-$ by the same picture, with the understanding that the intersections with integer homology spheres are capped off by Seifert surfaces.}\label{fig1}
\end{figure} 

The cobordism $W^i_j$ admits a family of metrics $G^i_j$ parametrized by the {\it associahedron} of dimension $i-j-1$. We need these families of metrics with $i-j\leq 3$ to define the homomorphisms used in the proof of Proposition \ref{prop:CDX-triangle}. We only focus on recalling the definition of these families of metrics within this range. This family will be constructed by stretching the metric along certain embedded submanifolds of $W^i_j$, depicted in Figure \ref{fig2} below in the case $i - j = 3$; the case $i - j = 2$ can be visualized by looking only at $2/3$ of this picture.\\

First fix Riemannain metrics on the 3-manifolds $Z_{1}$, $Z_0$ and $Z_{-1}$. (This choice is already assumed in the definition of the complexes $C_i$ in the previous subsection.) Fix Riemannian metrics $G^i_{i-1}$ on the cobordisms $W^1_0, W^0_{-1}, W^{-1}_{-2}$ with cylindrical ends modeled on the chosen metrics on $Z_{1}$, $Z_0$ and $Z_{-1}$ and the round metrics on $S^3$ and $\mathbb{RP}^3$. We extend these choices to arbitrary $i$ by requiring that $G^{k+3}_{k+2} = G^k_{k-1}$ for any $k$.

The family of metrics $G^i_{i-2}$ on $W^{i}_{i-2}$ is parametrized by the interval $[-1,1]$. Each end of the interval corresponds to a Riemannian metric which is broken along a separating codimension-$1$ submanifold. For a detailed reference discussing such families including broken metrics, see \cite[Section 5]{KMOS}. Since $W^{i}_{i-2}$ is the composition of the cobordisms $W^{i}_{i-1}$ and $W^{i-1}_{i-2}$, the chosen metrics on these cobordisms determine a metric on $W^{i}_{i-2}$, which is broken along an embedded copy of $Z_{i-1}$. This metric is the element of the family of metrics $G^i_{i-2}$ associated to the endpoint $1$ of the interval $[-1,1]$. For $1/3\leq t< 1$, let $T=1/(1-t)$. Then the metric corresponding to $t$ in the family of metrics is given by removing the subspaces $[T,\infty)\times Z_{i-1}$ from $W^{i}_{i-1}$ and $(-\infty,-T]\times Z_{i-1}$ from $W^{i-1}_{i-2}$ and then gluing the boundary components $Z_{i-1}$ together.

The metric corresponding to the other endpoint is given by a broken metric fully stretched along an embedded copy of $S^3$ or $\mathbb{RP}^3$ in $W^{i}_{i-2}$. First we identify the copy of $\mathbb{RP}^3$ in the cobordism $W^{1}_{-1}=X'\circ X''$, which is depicted as the curve crossing $Y$ in Figure \ref{fig2}. 

The cobordism $W^1_{-1}$ is given by gluing 
\[
  [-5,-3]\times S^1\times D^2,\hspace{1cm}[-1,1]\times S^1\times D^2,\hspace{1cm}[3,5]\times S^1\times D^2
\]
to $[-5,5]\times E(K)$ respectively using the gluing maps ${\bf 1}_{[-5,-3]} \times \varphi_{1}$, ${\bf 1}_{[-1,1]} \times \varphi_{0}$ and ${\bf 1}_{[3,5]} \times \varphi_{-1}$. (Analogous to $\varphi_{\pm1}$, the map $\varphi_{0}$ is a diffeomorphism $\partial (S^1\times D^2) \to \partial E(K)$, which in this case corresponds to $\infty$-surgery.) Then the union of following subspaces of $W^{1}_{-1}$ 
\begin{align}
  \{-4\}\times S^1\times D^2\subset [-5,-3]\times S^1\times D^2,\hspace{1cm}&\hspace{1cm}\{4\}\times S^1\times D^2\subset [3,5]\times S^1\times D^2,\\
  \gamma \times S^1\times S^1 \subset &[-5,5]\times N(\partial E(K)) \label{3rd-part}
\end{align}
gives a submanifold $M^{1}_{-1}$ of $W^{1}_{-1}$ homeomorphic to $\mathbb{RP}^3$. Here $N(\partial E(K))$ is a regular neighborhood of the boundary $\partial E(K)$ in $E(K)$, and in particular, it can be identified with $(-1,0]\times S^1\times S^1$. This gives an identification of $[-5,5]\times N(\partial E(K))$ with $[-5,5]\times (-1,0]\times S^1\times S^1$. In \eqref{3rd-part}, $\gamma$ is a properly embedded path in $[-5,5]\times (-1,0]$ whose endpoints are $(-4,0)$ and $(4,0)$. The curve $\gamma$ is chosen so that the resulting submanifold lies in the interior of $W^1_{-1}$.\\

The submanifold $M^{1}_{-1}$ is separating and the two connected components of the complement of a tubular neighborhood of $M^{1}_{-1}$ can be described as follows. One of the components can be regarded as a cobordism from $Z_1$ to $Z_{-1}$ with the middle end $\mathbb{RP}^3$, and is given by flipping around the orientation-reversal of $W^{-1}_{-2}$. In Figure \ref{fig2}, this is the region cobounded by the curves labeled $Z_1, \mathbb{RP}^3$, and $Z_{-1}$. 

The other connected component is a 4-manifold $N$ with three boundary components, one of which is $\mathbb{RP}^3$ and the other two are $S^3$. In Figure \ref{fig2}, $N$ is diffeomorphic to either of the two regions cobounded by curves labeled $\mathbb{RP}^3, S^3$, and $S^3$. Let $D$ be a twice punctured 3-ball with three boundary components $S_{-1}$, $S_0$ and $S_{1}$ which are diffeomorphic to the 2-sphere. Then $N$ is diffeomorphic to the total space of the $S^1$-bundle over $D$ whose Euler class evaluates to $-1$ on the boundary components $S_{\pm 1}$ of $D$ and to $2$ on the boundary component $S_0$. In particular, the boundary component $\mathbb{RP}^3$ of $N$ corresponds to the circle bundle over $S_0$. 

The metric corresponding to $-1$ in the family of metrics $G^i_{i-2}$ is given by a broken metric on $W^{1}_{-1}$, which is broken along $M^{1}_{-1}$. To be more precise, we fix metrics with cylindrical ends on the complement of $M^{1}_{-1}$, which has two cylindrical ends corresponding to the round metrics on $\mathbb{RP}^3$, two cylindrical ends corresponding to the round metric on $S^3$ and one cylindrical end for each of $Z_{\pm 1}$. Similar to the previous case, we define our family of metrics for the interval by $(-1,-1/3]$ by removing half cylinders from the $\mathbb{RP}^3$ ends, and then gluing the two connected components along their $\mathbb{RP}^3$ boundaries. Finally we extend our family of metrics on $W^{1}_{-1}$ to $(-1/3,1/3)$ in an arbitrary way so that all metrics $g_t$ coincide on the four ends of $W^1_{-1}$ with the chosen cylindrical end metrics.\\

The definitions of the families of metrics on $W^2_0$ and $W^0_{-2}$ follow a similar scheme. In the same way as in the previous case, we can form separating submanifolds $M^2_0\subset W^2_0$ and $M^0_{-2}\subset  W^0_{-2}$, which are in this case diffeomorphic to $S^3$. Then the family of metrics $G^2_0$ (resp. $G^0_{-2}$) for $t=1$ is given by a broken metric that is fully stretched along $Z_1$ (resp. $Z_{-1}$), for $t=-1$ is given by a broken metric that is fully stretched along $M^2_0$ (resp. $M^0_{-2}$) and is extended by (non-broken) metrics for $t\in (-1,1)$. Finally we extend this construction to any $W^i_{i-2}$ by requiring that $G^{k+3}_{k+1}=G^{k}_{k-2}$ for any $k$. We also remark that the complement of $M^i_{i-2}$ in $W^i_{i-2}$ has a similar description as before; one of the components is still diffeomorphic to $N$ and the other component is given by flipping the orientation-reversal of $W^{i+1}_i$.\\

\begin{figure}\label{fig2}
  \centering
\includegraphics[width=6in]{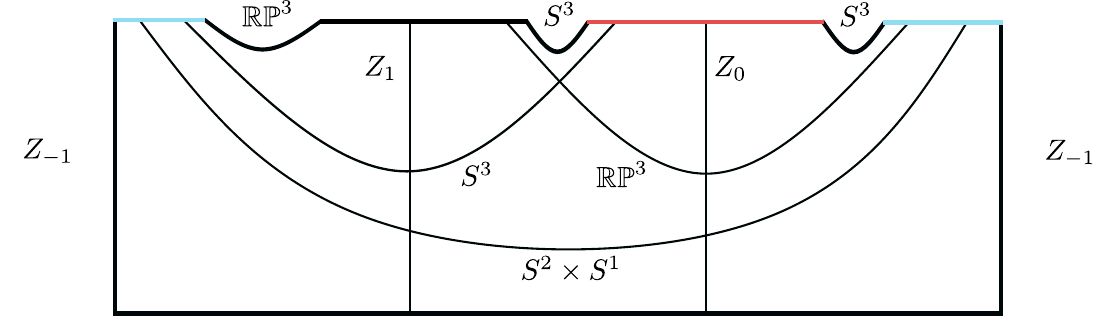}
\caption{A representation of the cobordism $W^2_{-1}$, with its five boundary components labeled. The more thin arcs represent the five submanifolds used in the construction of the family of metrics $G^2_{-1}$, and each arc is labeled by the diffeomorphism type of the corrresponding submanifold. The surface $\hat c^2_{-1}$ is represented as the lighter blue curve; the surface $\check c^2_{-1}$ is the union of the darker red and lighter blue curves. The middle end $\mathbb{RP}^3$ is drawn differently than the $S^3$ middle ends to make the asymmetry more apparent. The pictures for other $W^i_{i-3}$ are similar; for instance, the bundle data can on $W^i_{i-3}$ obtained by cyclically permuting the three depicted pieces.}\label{fig2}
\end{figure} 

Next, we turn into the description of the family of metrics $G^i_{i-3}$ on $W^i_{i-3}$ parametrized by a pentagon $P$, which is the 2-dimensional associahedron. First we consider the case of $G^2_{-1}$. We start by describing five codimension-$1$ submanifolds of the cobordism $W^2_{-1}=W^0_{-1} \circ W^1_0\circ W^2_1$. The definition of this cobordism as a composite implies that there are natural embeddings of $Z_0$ and $Z_1$ in $W^2_{-1}$. Since $W^2_{-1}=W^2_{0}\circ W^0_{-1}=W^2_{1}\circ W^1_{-1}$, there are also natural embeddings of $M^2_0$ and $M^1_{-1}$ into $W^2_{-1}$. To describe the fifth submanifold $M^{2}_{-1}$, note that $W^2_{-1}$ is given by gluing 
\[
  [-7,-5]\times S^1\times D^2,\hspace{1cm}[-3,-1]\times S^1\times D^2,\hspace{1cm}[1,3]\times S^1\times D^2,\hspace{1cm}[5,7]\times S^1\times D^2
\]
to $[-7,7]\times E(K)$ respectively using the gluing maps ${\bf 1}_{[-7,-5]} \times \varphi_{-1}$, ${\bf 1}_{[-3,-1]} \times \varphi_{1}$, ${\bf 1}_{[1,3]} \times \varphi_{0}$ and ${\bf 1}_{[5,7]} \times \varphi_{-1}$. Then $M^{2}_{-1}$ is a submanifold of $W^2_{-1}$ given as the union of the following three parts:
\begin{align*}
  \{-6\}\times S^1\times D^2\subset [-7,-5]\times S^1\times D^2,\hspace{1cm}&\hspace{1cm}\{6\}\times S^1\times D^2\subset [5,7]\times S^1\times D^2,\\
  \gamma \times S^1\times S^1 \subset &[-7,7]\times N(\partial E(K)), 
\end{align*}
where similar to \eqref{3rd-part}, $\gamma\subset [-7,7]\times (-1,0]$ is a properly embedded path with the endpoints $(-6,0)$ and $(6,0)$. This choice of $\gamma$ is to ensure both that $M^2_{-1}$ lies in the interior of $W^2_{-1}$, and to ensure that it is disjoint from both $M^1_{-1}$ and $M^2_0$.

The manifold $M^{2}_{-1}$ is diffeomorphic to $S^2\times S^1$. The complement of a tubular neighborhood of this submanifold of $W^2_{-1}$ has two connected components. One of the connected components, denoted $V_{-1}$, is given by removing from the product cobordism $[-7,7]\times Z_{-1}$ a tubular neighborhood of $\{0\}\times K_{-1}$; this is represented in Figure \ref{fig2} as the region bounded below the curve labeled $S^2 \times S^1$. The other connected component, denoted by $N'$, is given by removing from $N$ a tubular neighborhood of one of the fibers of the $S^1$ fibration of $N$ and is represented in Figure \ref{fig2} as the region bounded above the curve labeled $S^2 \times S^1$. (This discussion is a special case of the discussion of `spherical cuts' and their complements in \cite[Section 3.1]{CDX}.)

Label the edges of the pentagon $P$ cyclically by $e_j$ with $j\in \Z/5$, and label the common vertex of $e_j$ and $e_{j+1}$ with $v_{j,j+1}$. Then the edge $e_j$ of $P$ corresponds to the submanifold $Q_j$ of $W^2_{-1}$, where
\[
  Q_0=Z_0,\hspace{1cm}Q_1=Z_1,\hspace{1cm}Q_2=M^1_{-1},\hspace{1cm}Q_3=M^2_{-1},\hspace{1cm}Q_4=M^2_{0}.
\]
In particular, for any $j$ the submanifolds $Q_j$ and $Q_{j+1}$ are disjoint from each other. The metrics corresponding to the edge $e_j$ of $P$ are fully stretched along the submanifold $Q_j$. Furthermore, the endpoint $v_{j-1,j}$ (respectively $v_{j,j+1}$) of this edge corresponds to the metric that is also fully stretched respectively along $Q_{j-1}$ (respectively, $Q_{j+1}$). 

As we move from $v_{j-1,j}$ to $v_{j,j+1}$ the corresponding metrics vary from having large necks along $Q_{j-1}$ to small necks along $Q_{j-1}$ and then from small necks along $Q_{j+1}$ to large necks along $Q_{j+1}$, while all the metrics parametrized by this edge are broken along $Q_j$. In particular, we require that the metric parametrized by the edge $e_0$ is given by the metric $G^2_1$ on $W^2_{1}$ and the family of metrics $G^1_{-1}$ on $W^1_{-2}$. A similar requirement applies to the edge $e_1$. 

We extend the family of metrics from the boundary of $P$ first to a tubular neighborhood of the vertices $v_{j,j+1}$ in $P$ by replacing infinite length necks along $Q_j$ and $Q_{j+1}$ with finite length necks. Then we extend this family to a tubular neighborhood of $e_j$ by replacing the infinite length necks along $Q_j$ with finite length necks along this submanifold. Finally we extend this family to the rest of $P$ in a smooth way by non-broken metrics on $W^2_{-1}$, which are assumed to coincide with the chosen cylindrical metrics on the five ends of $W^2_{-1}$. (This is a special case of the construction of an associahedron of metrics in \cite[Section 4.3]{CDX}.)

\begin{figure}
  \centering
\includegraphics[width=6.5in]{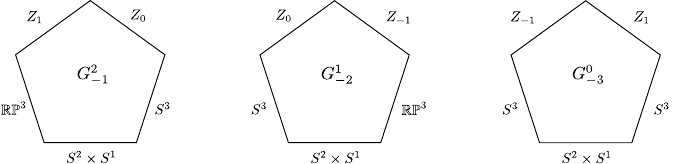}
\caption{The three pentagons of metrics. Each boundary edge corresponds to an interval of metrics broken along a given submanifold, and each edge in the diagram is labeled by the diffeomorphism type of the breaking submanifold. The vertex $v_{0,1}$ is the topmost vertex, and the edges $e_j$ proceed counter-clockwise.}\label{fig3}
\end{figure} 

The definition of the families of metrics $G^1_{-2}$ and $G^0_{-3}$ are similar, and we extend the construction of this family of metrics on $W^i_{i-3}$ to any $i$ by requiring that $G^{k+3}_{k}=G^{k}_{k-3}$ for any $k$. In particular, as a part of the construction of the family of metrics, we from a submanifold $M^{i}_{i-3}$ of $W^i_{i-3}$ which is diffeomorphic to $S^2\times S^1$, and the complement of its tubular neighborhood has two connected components. One of the connected components is again diffeomorphic to $N'$ and the other connected component, denoted by $V_i$, is the complement of a tubular neighborhood of $\{0\}\times K_{i}$ in $[-7,7]\times Z_{i}$.\\

We will later need both an explicit understanding of the cohomology of these cobordisms and a description of certain embedded surfaces. The 4-manifolds $X$, $X'$ and $X''$ all have trivial first cohomology, and their second cohomology groups are isomorphic to $\Bbb Z$. In the description of $X$ above, the cylinders 
\[
\mathfrak c_-:=[-3,-1]\times S^1\times \{0\}\subset [-3,-1]\times S^1\times D^2,\hspace{0.5cm} \mathfrak c_+:=[1,3]\times S^1\times \{0\}\subset [1,3]\times S^1 \times D^2
\] 
determine relative homology classes for $(X,\partial X)$, and the Poincar\'e dual of each of these determines a generator for the second cohomology group of $X$. We orient these cylinders using the product orientation. The intersection of $\mathfrak c_\pm$ with $Z_{\pm 1}$ is the dual knot of the Dehn surgery, and we can glue a Seifert surface of this knot to $\mathfrak c_\pm$ to obtain an oriented properly embedded surface $c_{\pm}$ representing the same cohomology class as $\mathfrak c_\pm$ in $X$. Note that the only boundary component of $c_\pm$ is in $\mathbb{RP}^3$ and the restriction of the cohomology class of $c_{\pm}$ to this middle end is the generator of $H^2(\mathbb{RP}^3)$. We may similarly form embedded oriented surfaces $c_\pm'\subset X'$ and $c_\pm''\subset X''$ which represent generators of the second cohomologies of $X'$ and $X''$. 

The intersection form of $X$ is positive-definite and the intersection forms of $X'$ and $X''$ are negative definite. Because each of $X, X', X''$ has boundary a union of rational homology spheres, a second cohomology class has a well-defined square $c^2 \in \Bbb Q$, defined using the isomorphism $H^2(X, \partial X; \Bbb Q) \to H^2(X; \Bbb Q)$. For the cohomology classes above, we have $c_\pm^2 = 1/2$, while $(c'_\pm)^2 = (c''_\pm)^2 = -1$. 

We define a pair of embedded surfaces $\hat c^i_j, \check c^i_j \subset W^i_j$ by the formulas \begin{align*}
\hat c^1_0 = \varnothing \subset W^1_0 \quad \quad &\hat c^0_{-1} = c_+ \subset W^0_{-1}, \quad \quad \quad\quad \hat c^{-1}_{-2} = c_- \subset W^{-1}_{-2}, \\
\check c^1_0 = c_+ \subset W^1_0, \quad \quad &\check c^0_{-1} = c_- \sqcup c_+ \subset W^0_{-1}, \quad \quad \check c^{-1}_{-2} = c_- \subset W^{-1}_{-2}
\end{align*}
and mod $3$ periodicity in the case that $j = i-1$. For general $i > j$, we define $\hat c^i_j = \hat c^i_{i-1} \sqcup \cdots \sqcup \hat c^{j+1}_j$, and similarly for $\check c^i_j$. This is a special case of the more general construction of \cite[Section 3.2]{CDX}. The manifold $N$ contains the embedded surfaces \begin{equation}\label{cN}
\hat c_N := \hat c^1_{-1} \cap N, \hspace{1cm}\check c_N := \check c^1_{-1} \cap N,
\end{equation} 
and the Poincar\'e dual cohomology classes satisfy $\hat c_N\cdot \hat c_N=\check c_N \cdot \check c_N = -1/2$. 

\subsection{Reducibles}\label{sec:reducibles}
The maps $f_i, g_i, h_i$ will be defined by counting isolated irreducible instantons on $W^i_j$ with respect to the family of metrics $G^i_j$. The relations for these maps will be computed by studying the ends of the 1-dimensional moduli spaces of irreducible instantons. 

To ensure that these counts are defined, we need the 0-dimensional moduli spaces of irreducible instantons to be compact. Similarly, to get the expected boundary relations, we need some control over the reducible instantons that appear in 1-dimensional moduli spaces. There was no such issue in \cite{CDX}, where the 3-manifolds were equipped with \emph{admissible bundles}, which support no reducible flat connections. In this subsection, we will analyze the reducible instantons with respect to the families of metrics $G^i_j$. Using the study of such reducibles, we will show in the following subsection that the argument of \cite{CDX} goes through with minimal change, by showing that the only reducibles which appear in the compactification of the moduli spaces of interest are precisely those used in the proof of \cite[Theorem 1.6]{CDX}. 

First, we review the definition of the relevant moduli spaces of instantons. Previously, we introduced properly embedded oriented surfaces $\hat c^i_j$ and $\check c^i_j$ in $W^i_j$, which are abbreviated to $c$ for now. For flat connections $\alpha$ on $Z_i$ and $\beta$ on $Z_j$, there is a moduli space $M(W^i_j,  c, \text{int}(G^i_j); \alpha, \beta)$ of pairs $(g, A)$, where $g \in \text{int}(G^i_j)$ and $A$ is a $g$-ASD connection on $(W^i_j, c)$ which is asymptotic to $\alpha$ on $Z_i$ and $\beta$ on $Z_j$; these pairs are considered up to gauge equivalence. 

To first extend this to a moduli space $M(W^i_j, c, G^i_j; \alpha, \beta)$, we include instantons with respect to the broken metrics. Suppose $g \in G^i_j$ is a broken metric, broken along a single connected submanifold $Y$ (so $W^i_j = W \cup_Y W'$, with $Y$ being the end intermediate to these). Then an instanton with respect to this broken metric is a pair $(A, A')$ of an instanton on $W$ and an instanton on $W'$ with respect to the relevant cylindrical-end metrics, so that the $A$ and $A'$ are asymptotic to the same flat connection along $Y$, now considered up to simultaneous gauge equivalence.

The definition is similar for metrics broken along more than one connected submanifold. We define $M(W^i_j, c, G^i_j; \alpha, \beta)$ to be the moduli space of pairs $(g, A)$, where $g \in G^i_j$ and $A$ is a $g$-instanton on $(W^i_j, c)$ asymptotic to $\alpha$ on $Z_i$ and $\beta$ on $Z_j$, considered up to gauge equivalence; when $g$ is broken, this is understood in the generalized sense described above.  Here $\alpha$ and $\beta$ are perturbed flat connections given as the critical points of the perturbed Chern--Simons functional associated to $Z_i$ and $Z_j$. We topologize this with respect to convergence on compact subsets in the complement of the broken submanifolds.

This moduli space remains noncompact. Just as we pass from $M(W, c, \text{int}(G))$ to $M(W, c, G)$ by introducing instantons with respect to these broken metrics, there exists a compactification $M^+(W, c, G; \alpha, \beta)$ which introduces \emph{broken solutions} to the ASD equations. In the statement below, if $A$ is an instanton on a cobordism $W$, the \emph{positive limit} of $A$ is the flat connection $A$ is asymptotic to on its outgoing boundary component; similarly with \emph{negative limit} and \emph{middle limit}. 

\begin{definition}\label{def:broken-reducible}
Let $(W,c): Y \xrightarrow{L} Y'$ be a cobordism with a middle end. Suppose $W$ is equipped with an unbroken Riemannian metric $g$. A \textbf{(possibly) broken instanton} on $(W,c)$ is a sequence 
\[A = (B_1, \cdots, B_n,A_W, C_1, \cdots, C_m, B'_1, \cdots, B'_\ell)\] of connections satisfying the following conditions:
\begin{enumerate}[label=(\alph*)]
\item The connection $B_i$  (resp. $B'_i, C_i$) is a nonconstant instanton on $\Bbb R \times Y$ (resp. $\Bbb R \times Y'$, $\Bbb R \times L$) considered modulo translation, while $A_W$ is an instanton on $(W,c)$. 
\item For $1 \le i < n$, the positive limit of $B_i$ is equal to the negative limit of $B_{i+1}$, and similarly for $C_i$ and $B'_i$.
\item The negative limit of $A_W$ is the positive limit of $B_n$, while the middle limit of $A_W$ is the negative limit of $C_1$ and the positive limit of $A_W$ is the negative limit of $B'_1$. 
\end{enumerate}
The \textbf{outer limits} of a broken instanton are given by the negative limit of $B_1$ and the positive limit of $B'_\ell$. If every instanton in the sequence is irreducible and positive and negative limiting flat connections of $A_W$, $B_i$ and $B_i'$ are all irreducible, we say this broken instanton is \textbf{fully irreducible}.
\end{definition}

A similar definition holds in the case that $W$ is equipped with a broken metric $g$. We will say `instanton' to refer to a possibly-broken instanton, say `unbroken instanton' when we want to assume $n = m = \ell = 0$ (including in the case that $g$ is a broken metric), and say `broken instanton' for an instanton with $n+m+\ell > 0$. The \emph{intermediate limits} of a broken instanton are the $n+m+\ell$ flat connections obtained as the positive limits of $B_j$ and negative limits of $B'_j, C_j$.

Let $\beta$ be a connection given as the critical point of the perturbed Chern--Simons functional of a 3-manifold $Y$. There are two non-negative integers associated to $\beta$. The first one, denoted $h^0_\beta$ is $\dim \Gamma_\beta$, is the dimension of the stabilizer of $\beta$ under the gauge group action. When $\beta$ is irreducible, this is zero; when $\beta$ has central holonomy this dimension is three, and when $\beta$ has abelian but not central holonomy this dimension is one. The second integer $h^1_\beta$ is the dimension of the null space of the Hessian of the perturbed Chern--Simons functional at $\beta$. In the case that the perturbation is trivial, $h^1_\beta$ agrees with the dimension of the Zariski tangent space of the corresponding character variety at $\beta$. We say $\beta$ is non-degenerate if $h^1_\beta=0$. Most of the connections that we encounter below are non-degenerate; the only exception occurs in Subsection \ref{hi-const}. Because of this, we assume that all limits of instantons appearing in the rest of this subsection are non-degenerate.

We write $r(A)$ for the sum of $h^0_\beta$ over all intermediate limits $\beta$ of $A$; notice that $r(A) > 0$ if and only if some intermediate limit of $A$ is reducible.  

\begin{definition}
Let $A$ be an instanton with non-degenerate limits. Then the \textbf{index} of $A$ is defined to be $$i(A) = r(A) + i(A_W) + \sum_{j=1}^n i(B_j) + \sum_{j=1}^m i(C_j) + \sum_{j=1}^\ell i(B'_j),$$ where $i$ is the index of the ASD operator. 
\end{definition}

This coincides with the index of the connection obtained by gluing the constintuent instantons of $A$ into a connection on $W$. These definitions are relevant as follows. We define the space $M^+(W, c, G; \alpha, \beta)_d$ to be the space of pairs $(g, A)$, where $g \in G$ is a (possibly broken) metric and $A$ is  a (possibly broken) instanton on $(W, c)$ with index $d - \dim G$, considered up to gauge equivalence. This is given the topology of `chain-convergence', as in \cite[Section 5.1]{DonBook}. This moduli space has expected dimension $d$. It is established as  \cite[Proposition 5.5]{DonBook} that if $M^+(W, c, G; \alpha, \beta)_{d-8n}$ is empty for all $n > 0$, the moduli space $M^+(W, c, G; \alpha, \beta)_d$ is compact.

We may choose a small perturbation of the ASD equation with the following properties: it is zero on the middle ends and equal to a chosen small perturbation on the other ends, and all unbroken irreducible instantons over $G$ are cut out transversely in this family. Thus any fully irreducible instanton on $(W,c,G)$ has index $i(A) \ge -\dim G$, with equality if and only if $A$ is an unbroken instanton supported by an unbroken metric $g \in \text{int}(G)$. The cobordisms we use either have $b^+(W) = 0$ or are equipped with a non-trivial admissible bundle, so central connections are cut out transversely and abelian ASD connections are cut out transversely within the reducible locus. Thus we may also choose our perturbation to vanish near the central connections and within the reducible locus, so that the set of reducible ASD connections is unchanged by perturbation.

We will define our maps by counting the number of elements in $M(W^i_j,c, G^i_j; \alpha, \beta)_0$ for $c = \hat c^i_j$ or $c = \check c^i_j$. To ensure these counts are well-defined, we will need to ensure that this space coincides with its compactification; that is, that there are no broken instantons of index $i(A) = -\dim G$. Because this is automatic when $A$ is fully irreducible, our goal is to rule out the possibility of reducibles.

When verifying the boundary relations for these maps, we will investigate the moduli space $M(W, c, G; \alpha, \beta)_1$. Our boundary relations mostly involve counts of fully irreducible broken instantons, and we therefore need to rule out the possibility of small index instantons which are not fully irreducible.
Because reducible flat connections on a cylinder are constant, there are two ways a connection could fail to be fully irreducible: 

\begin{itemize}
\item It could be possible that all components of $A$ are irreducible, but some interior limit is reducible; 
\item The connection $A_W$ could be reducible.
\end{itemize}

The first type is easy to rule out.

\begin{lemma}\label{lemma:no-irred-only}
Suppose $A$ is a broken instanton on $(W^i_j, \hat c^i_j)$ or $(W^i_j, \check c^i_j)$ with respect to a metric in $\textup{int}(G^i_j)$. If the constituent instantons are all irreducible but $A$ is not fully irreducible, then $i(A) \ge 2 - (i - j - 1)$. 
\end{lemma}

\begin{proof}
As discussed above, all irreducible instantons on $(W^i_j, c, G^i_j)$ have $i(A_W) \ge - \dim G^i_j = -(i-j-1)$, and all irreducible instantons on a cylinder have $i(B) \ge 1$. Any broken instanton has at least one component which is an instanton on a cylinder, and because $A$ is not fully reducible we have $r(A) > 0$, so $i(A) \ge 1 + 1 -(i-j-1)$. 
\end{proof}

As for the second type, observe that if $W: Y \xrightarrow{L} Y'$ is a cobordism between integer homology spheres (with $L$ a disjoint union of $S^3$s and $\Bbb{RP}^3$s) and $A_W$ is reducible, then the positive and negative limits of $A_W$ are also reducible, hence central because $Y, Y'$ are integer homology spheres. Therefore, any broken instanton with irreducible outer limits for which $A_W$ is reducible has at least two components which are instantons on cylinders, as well as two central interior limits, so $$i(A) \ge 1 + 3 + i(A_W) + 3 + 1 = 8 + i(A_W).$$ It suffices to give a suitable lower bound on the minimal index of such $A_W$. 

\begin{lemma}\label{lemma:min-index}
If $A_W$ is a reducible unbroken instanton on $(W^i_j, \hat c^i_j)$ or $(W^i_j, \check c^i_j)$ with respect to a metric in $G^i_j$, then $i(A_W) \ge -4$.
\end{lemma}
\begin{proof}
In our case, where $b_1(W) = 0$, the index formula of \cite[Equation (82)]{MMR} simplifies to the following equation, where we write $L = L_1 \sqcup \cdots \sqcup L_n$ with $A_W|_{L_i} = \alpha_i$:
\begin{equation}\label{eqn:index}
i(A_W) = 8\mathcal E(A_W) - 3(1 + b^+(W)) + \frac 12 \sum_{i=1}^n (3 - h_{\alpha_i} + \rho(\alpha_i)).
\end{equation} 
The quantity $h_{\alpha_i}$ is $3$ when $\alpha_i$ is central, $1$ when $\alpha_i$ is abelian and $0$ when $\alpha_i$ is irreducible. For us, $(L_i, w_i)$ are all either $(S^3, \varnothing)$ or $(\mathbb{RP}^3, w)$ for $[w] \neq 0\in H_1(\mathbb{RP}^3; \mathbb Z/2)$, and the quantity $\rho(\alpha_i)$ is zero for any flat connection on $S^3$ or $\Bbb{RP}^3$ by explicit computation, e.g. \cite[Proposition 2.12]{APS2}. 

The manifold $S^3$ supports only the central flat connection, and $(\mathbb{RP}^3, w)$ with $[w] \ne 0 \in H_1(\mathbb{RP}^3; \Bbb Z/2)$ supports exactly one abelian flat connection. In particular, the sum over $i$ gives twice the number of $\mathbb{RP}^3$ middle ends.

We now investigate the energy term, which we claim is nonnegative and strictly positive when $W$ has an $\mathbb{RP}^3$ middle end. 

As $A_W$ is an instanton, we have $\mathcal E(A_W) \ge 0$ with equality if and only if $A_W$ is projectively flat; equivalently, if and only if the induced connection on the adjoint $SO(3)$-bundle is flat. If $A_W$ is reducible, there exists a parallel splitting $E_c \cong L_x \oplus L_y$, where $x = c_1(L_x)$ and therefore $PD(c) = x+y$. The induced connection on the adjoint bundle respects a splitting $\text{ad}(E_c) \cong \mathbb R \oplus L'$, where $L' = L_x \otimes_{\mathbb C} L_y^{-1}$. Because this connection is flat, it follows that $L'$ is a flat bundle, and hence that $2x - PD(c) = x-y$ is a torsion cohomology class on $W$. However, when one of our manifolds $W$ has a $\mathbb{RP}^3$ end, the quantity $2x - PD(c)$ is never torsion, as $PD(c)$ is a primitive class.

For the manifolds $W^i_j$ with $i-j \le 3$, either $b^+(W) = 0$ or we have that $b^+(W) = 1$ and $L$ is a disjoint union of $\Bbb{RP}^3$ and some number of copies of $S^3$ with $c|_{\Bbb{RP}^3}$ is non-trivial. In the former case, the index computation gives $i(A_W) \ge -3$. In the latter case, the index computation gives $i(A_W) > -6 + 1 = -5$ because $\mathcal E(A_W)$ is positive.
\end{proof}

The previous two lemmas will be the main way we \emph{rule out} broken instantons which are not fully irreducible. However, there are some places in the argument of \cite{CDX} where reducible connections do enter in an essential way. In particular, we will need to understand the reducibles for the pairs $(N, \hat c_N)$ and $(N, \check c_N)$, where $\hat c_N$ and $\check c_N$ are given in \eqref{cN}.

\begin{lemma}\label{lemma:N-reducibles}
Up to the action of the gauge group, there is a unique minimal-index reducible instanton $\hat A_N$ on $(N,\hat c_N)$, for which $i(\hat A_N) = -1$. Similarly, there is a unique minimal-index reducible instanton $\check A_N$ on $(N,\check c_N)$ which satisfies $i(\check A_N) = -1$.
\end{lemma}
\begin{proof}
By Hodge theory and the fact that $b^+(N) = 0$, reducible solutions to the ASD equation on $(N,\hat c_N)$ are in bijection with pairs $\{x,y\} \subset H^2(N;\Bbb Z)$ for which $x+y = PD(\hat c_N)$. Because $H^2(N;\Bbb Z) \cong \Bbb Z$ and $\hat c_N$ is a generator, there is one reducible solution $\{n\hat c_N, (1-n)\hat c_N\}$ for each integer $n \ge 1$.

As discussed above, the quantity $\mathcal E(A_W)$ in (\ref{eqn:index}) for such a reducible is precisely $-2(x-y)^2$. The surface $\hat c_N$ represents a generator of $H^2(N; \Bbb Z)$, for which $\hat c_N^2 = -1/2$. Any pair $\{n\hat c_N, (1-n)\hat c_N\}$ gives $-2(x-y)^2 = (2n-1)^2$, which is minimized for $n=1$. There is thus one minimal-index reducible, which we denote by $\hat A_N$. Because $(N, \hat c_N)$ has boundary $(S^3, \varnothing) \sqcup (S^3, \varnothing)\sqcup (\mathbb{RP}^3,w)$, we see that $\hat A_N$ restricts trivially to two boundary components, and to an abelian connection on the third. Therefore, (\ref{eqn:index}) gives $$i(\hat A_N) = 1 -3 + 1 = -1.$$
This completes the proof for $(N,\hat c_N)$. The same argument applies to $(N,\check c_N)$. 
\end{proof}

\subsection{Proof of Proposition \ref{prop:CDX-triangle}}\label{sec:proof-36}
Here we establish the existence of the maps $f_i, g_i, h_i$, as well as their boundary relations. By the discussion of Section \ref{sec:triangle-detection}, this gives us the desired exact triangle. 

The general principle is as follows. If $(W,c): Y \to Y'$ is a cobordism equipped with a family of metrics $G$, the solutions to the $G$-ASD equation are cut out transversely and there are no reducible instantons with index less than $-\dim G$, then counting instantons over $(W,c,G)$ with index equal to $-\dim G$ gives rise to a well-defined map $C_*(Y) \to C_*(Y')$ satisfying the relation $$d_{Y'}\circ (W,c,G)_* + (W,c,G)_*\circ d_Y = (W, c,\partial G)_*.$$ The relation is derived by inspecting the ends of the moduli space of instantons $M^+(W, c, G; \alpha, \beta)_1$ with $\alpha$ and $\beta$ irreducible.

Our task is to verify that instantons supported by the interior of each face of $G$ are cut out transversely and to compute the map $(W, c,\partial G)_*$.

\subsubsection{The maps $f_i$}
Each manifold $W^i_{i-1}$ is equipped with an unbroken metric. The manifold $W^2_1$ has $b^+(W) > 0$ and a geometric representative $c \subset W^2_1$ for a non-trivial admissible $U(2)$-bundle. A combination of Lemmas \ref{lemma:no-irred-only} and \ref{lemma:min-index} imply that the moduli spaces of irreducible instantons on $W^i_{i-1}$ of index $0$, with irreducible outer limits, are compact, and thus consists of finitely many points. 

We can therefore define 
\begin{align*}f_2 &= (W^2_1, G^2_1, \hat c^2_1)_*: C_*(Z_{-1}) \to C_*(Z_1),  \\
f_1 &= (W^1_0, G^1_0, \hat c^1_0)_* \oplus (W^1_0, G^1_0, \check c^1_0)_*: C_*(Z_1) \to C_*(Z_0) \oplus C_*(Z_0) \\
f_0 &= (W^0_{-1}, G^0_{-1}, \hat c^0_{-1})_* \oplus -(W^0_{-1}, G^0_{-1}, \check c^0_{-1}): C_*(Z_0) \oplus C_*(Z_0) \to C_*(Z_{-1})\end{align*}
to be the maps which count index zero irreducible instantons with irreducible outer limits on the relevant cobordisms. 

The moduli spaces of index $1$ consists of fully irreducible solutions by Lemmas \ref{lemma:no-irred-only} and \ref{lemma:min-index}, where standard gluing techniques apply. We therefore obtain the following result.

\begin{proposition}\label{prop:f}
The maps $f_i$ are well-defined chain maps.
\end{proposition}

\subsubsection{The maps $g_i$}
Now we consider the 1-parameter family of metrics $G^i_{i-2}$ on $W^i_{i-2}$. We count instantons with irreducible outer limits and of index $-1$ with respect to this family to define the maps $g_i$, and use the moduli space of index $0$ instantons with irreducible outer limits to prove the desired chain homotopy relation. The combination of Lemmas \ref{lemma:no-irred-only} and \ref{lemma:min-index} imply that for any \emph{unbroken} metric in $\text{int}(G^i_{i-2})$, these moduli spaces consist of fully irreducible connections, but the broken metrics require somewhat more care. In fact, it is important to the argument that one boundary point of $G^1_{-1}$ \emph{does} support  an broken instanton with a reducible component. 

\begin{lemma}\label{lemma:g-bdry0}
Suppose $A$ is an instanton with irreducible outer limits on $(W^i_{i-2}, \hat c^i_{i-2})$ or $(W^i_{i-2}, \check c^i_{i-2})$ with respect to the metric broken along $Z_{i-1}$. If $A$ is not fully irreducible, then $i(A) \ge 1$. 
\end{lemma}
\begin{proof}
Suppose $A^i_{i-1}$ and $A^{i-1}_{i-2}$ are the connections on $W^i_{i-1}$ and $W^{i-1}_{i-2}$ induced by $A$. 
If $A$ has irreducible components but some reducible intermediate flat connection, then each constituent instanton has non-negative index and $r(A)\geq 1$, so $i(A) \ge 1$. 
If $A^i_{i-1}$ is reducible but $A^{i-1}_{i-2}$ is not, then we have $$i(A) \ge 1 + 3 + i(A^i_{i-1}) + 3 + i(A^{i-1}_{i-2}),$$ as $A^i_{i-1}$ must be glued to some trajectory on the cylinder to have irreducible outer limits. Applying Lemma \ref{lemma:min-index} and the fact that $i(A^{i-1}_{i-2}) \ge 0$, we see that $i(A) \ge 3$. The same argument applies to the case that $A^{i-1}_{i-2}$ is reducible but $A^i_{i-1}$ is not.
Finally, if both the connections $A^i_{i-1}$ and $A^{i-1}_{i-2}$ are reducible, then Lemma \ref{lemma:min-index} implies that 
\[i(A) \ge 1+3+(-4)+3+(-4)+3+1 = 3. \qedhere\] 
\end{proof}

The other endpoint of $G^i_{i-2}$ corresponds to a decomposition $W^i_{i-2} = (-W^{i+1}_{i}) \cup_{-L_{i+1}} N$. The index $i$ determines whether $-W^{i+1}_{i}$ has middle end diffeomorphic to $S^3$ or $\mathbb{RP}^3$, with the latter corresponding to the case $i \equiv 1 \mod 3$. 

\begin{lemma}\label{lemma:g-bdry1}
Suppose $A$ is an instanton with irreducible outer limits on $(W^i_{i-2}, \hat c^i_{i-2})$ with respect to the metric broken along $M^i_{i-2}$. Suppose $A$ is not fully irreducible. If $i \not\equiv 1 \mod 3$, then $i(A) \ge 2$. If $i \equiv 1 \mod 3$, then $i(A) \ge 0$ with equality if and only if $A = (A_{W}, \hat A_N)$ for $A_{W}$ an unbroken irreducible instanton of index zero on $W=-W^{i+1}_i$ and $\hat A_N$ the unique reducible of minimal index on $(N, \hat c_N)$ (see Lemma \ref{lemma:N-reducibles}). A similar claim holds for $(W^i_{i-2}, \check c^i_{i-2})$.
\end{lemma}

\begin{proof}
Suppose $A$ is an instanton on $(W^i_{i-2}, \hat c^i_{i-2})$ satisfying the assumption. We first assume that $A_W$, the component of $A$ on $-W^{i+1}_i$, is reducible. The discussion before Lemma \ref{lemma:min-index} implies that 
\[i(A) \ge 8 + i(A_{W}) + i(A_N) + h^0_\alpha,\] 
where $h^0_\alpha$ is the dimension of the stabilizer of the flat connection on $M^i_{i-2}$. By Lemma \ref{lemma:min-index} and Lemma \ref{lemma:N-reducibles}, we see that in this case $i(A) \ge 4$. Thus we may assume that $A_{W}$ is irreducible. We may also assume that $A$ contains no components given by instantons on a cylinder, as these increase the index by at least one. Thus, we may suppose $A = (A_{W}, A_N)$, for which the index is $i(A) = i(A_{W}) + i(A_N) + h^0_\alpha$. 

Because $A_{W}$ is irreducible, $i(A_{W}) \ge 0$. By Lemma \ref{lemma:N-reducibles}, we have $i(A_N) \ge -1$ with equality if and only if $A_N=\hat A_N$. Finally, $h^0_\alpha = 1$ if $i \equiv 1 \mod 3$ and $h^0_\alpha = 3$ otherwise. Thus if $i \equiv 1 \mod 3$ we have $i(A) \ge 0$ with equality if and only if $i(A_{W}) = 0$ and $i(A_N) = -1$, whereas if $i \not\equiv 1 \mod 3$ we have $i(A) \ge 2$. 
\end{proof}

Now we define maps 
\begin{align*}
g_2 &:= (W^2_0, G^2_0, \hat c^2_0)_* \oplus (W^2_0, G^2_0, \check c^2_0)_*: C_*(Z_{-1}) \to C_*(Z_0) \oplus C_*(Z_0)\\
g_1 &:= (W^1_{-1}, G^1_{-1},\hat c^1_{-1})_* \;-\; (W^1_{-1}, G^1_{-1}, \check c^1_{-1})_*: C_*(Z_1) \to C_*(Z_{-1}) \\
g_0 &:= (W^0_{-2}, G^0_{-2}, \hat c^0_{-2})_* \oplus (W^0_{-2}, G^0_{-2}, \check c^0_{-2})_*: C_*(Z_0) \oplus C_*(Z_0) \to C_*(Z_1)
\end{align*}
by counting instantons with irreducible outer limits and index $-1$ in the respective family of metrics and perturbations, with the approprsiate bundle. By a combination of Lemmas \ref{lemma:no-irred-only}, \ref{lemma:min-index} and Lemmas \ref{lemma:g-bdry0}, \ref{lemma:g-bdry1} above, we see that this moduli space is compact, so the maps $g_i$ are well defined. The boundary relations are slightly more subtle. 

\begin{proposition}\label{prop:g}
The maps $g_i$ defined above satisfy the relations $$d_{i-2} g_i + f_{i-1}  f_i + g_i  d_i = 0.$$
\end{proposition}
\begin{proof}
First we focus on $g_2$. The same argument will apply to $g_0$ with essentially no change. The desired relation decomposes as a direct sum of two relations involving the two components of $g_2$ and the two components of $f_1f_2$. More precisely, if $A$ is a broken instanton for a metric in $G^i_{i-2}$ with $i(A) \le 0$, then $A$ is fully irreducible by the discussion of Subsection \ref{sec:reducibles}, Lemmas \ref{lemma:g-bdry0} and \ref{lemma:g-bdry1}. Thus one finds by the usual gluing techniques that $$d_0 (W^2_0, G^2_0, \hat c^2_0)_* + (W^2_0, G^2_0, \hat c^2_0)_*d_2 = (W^2_0, \partial G^2_0, \hat c^2_0)_*,$$ and similarly for $\check c^2_0$. The map induced by the incoming boundary point of $\partial G^2_0$ is the composite map $(W^1_0, G^1_0,\hat c^1_0)_* \circ (W^2_1, G^2_1,\hat c^2_1)_*$ of the first component of $f_1 f_2$. By Lemma \ref{lemma:g-bdry1}, the other boundary component induces the zero map.  Running the same argument for $\check c^2_0$ and taking a direct sum, we obtain the relation $d_0 g_2 + f_1 f_2 + g_2 d_2 = 0$. 

As for $g_1$, the strategy follows \cite[Proposition 5.14(i)]{CDX}. Arguing in the same way as in the previous case, we obtain the following two relations:
\begin{align}
d_{-1}(W^1_{-1}, G^1_{-1},\hat c^1_{-1})_*+(W^1_{-1}, G^1_{-1},\hat c^1_{-1})_*d_1&=(W^1_{-1}, \partial G^1_{-1},\hat c^1_{-1})_*\label{g1-rel-1}\\
d_{-1}(W^1_{-1}, G^1_{-1},\check c^1_{-1})_*+(W^1_{-1}, G^1_{-1},\check c^1_{-1})_*d_1&=(W^1_{-1}, \partial G^1_{-1},\check c^1_{-1})_*\label{g1-rel-2}
\end{align}
The contribution induced by the incoming boundary point of $\partial G^1_{-1}$ in the maps $(W^1_{-1}, \partial G^1_{-1},\hat c^1_{-1})_*$ and $(W^1_{-1}, \partial G^1_{-1},\check c^1_{-1})_*$ are respectively equal to 
\[
  (W^0_{-1}, G^0_{-1}, \hat c^0_{-1})_* (W^1_0, G^1_0, \hat c^1_0)_*\hspace{1cm} (W^0_{-1}, G^0_{-1}, \check c^0_{-1})_* (W^1_0, G^1_0, \check c^1_0)_*
\]
Unlike the case of $g_2$, the contributions induced by the other boundary point of $\partial G^1_{-1}$, corresponding to the metric broken along $\mathbb{RP}^3$, are not automatically zero. Rather, these maps are equal to each other as we explain momentarily. In particular, after taking the difference of \eqref{g1-rel-1} and \eqref{g1-rel-2}, we obtain  
\[
  d_{-1}g_1+g_1d_{1}= (W^0_{-1}, G^0_{-1}, \hat c^0_{-1})_* (W^1_0, G^1_0, \hat c^1_0)_* - (W^0_{-1}, G^0_{-1}, \check c^0_{-1})_* (W^1_0, G^1_0, \check c^1_0)_*.
\]
The right hand side is equal to $f_0 f_1$, and this proves the claim.

To see that the maps induced by the bundles $\hat c^1_{-1}$ and $\check c^1_{-1}$ corresponding to the metrics broken along $\mathbb{RP}^3$ are the same, observe that $$\hat c^1_{-1} \cap (-W^{-1}_{-2}) = - \hat c^{-1}_{-2} = -\check c^{-1}_{-2} = \check c^1_{-1} \cap (-W^{-1}_{-2}),$$ while the intersections of $\hat c^1_{-1}$ and $\check c^1_{-1}$ with $N$ agree with $\hat c_N$ and $\check c_N$. Now by Lemma \ref{lemma:g-bdry1}, with respect to the metric broken along $\mathbb{RP}^3$, the only broken instantons of index zero on $(W^1_{-1}, \hat c^1_{-1})$ have $A_{W}$ an irreducible unbroken instanton of index zero, and $\hat A_N$ the unique minimal index reducible.  Thus, for $\hat c$ this map coincides on the chain level with the map $(-W^{-1}_{-2}, -\hat c^{-1}_{-2})_*$, and a similar claim holds for $\check c$. As discussed above $\hat c^{-1}_{-2} = \check c^{-1}_{-2}$, so indeed the two maps are the same. 
\end{proof}

\subsubsection{The maps $h_i$}\label{hi-const}

We move on to the maps $h_i$ and their boundary relations. These are defined by counting index $-2$ irreducible instantons on the cobordisms $W^i_{i-3}$ with respect to the family of metrics $G^i_{i-3}$. For this family,  
\begin{itemize}
\item we need to show there are no unbroken instantons of index less than $-2$ supported by metrics in the interior of $G^i_{i-3}$; 
\item we need to show there are no unbroken instantons of index less than $-1$ supported by $\partial G^i_{i-3}$;
\item we need to determine the induced map of $\partial G^i_{i-3}$.
\end{itemize}

Because there are three families $G^i_{i-3}$ and each family contains five boundary strata, this involves a great deal of case analysis. The following four lemmas in our case analysis follow exactly as in the previous section, and we omit their proofs.

The first lemma asserts that there are no small-index reducibles on the interior of the family. This lemma follows from Lemma \ref{lemma:no-irred-only}.

\begin{lemma}\label{lemma:h1}
Suppose $A$ is an instanton with irreducible outer limits on $(W^i_{i-3}, \hat c^i_{i-3})$ or $(W^i_{i-3}, \check c^i_{i-3})$ with respect to one of the metrics in $\textup{int}(G^i_{i-3})$. Then $i(A) \ge -2$ with equality if and only if $A$ is unbroken and hence fully irreducible.
\end{lemma}

The second lemma asserts that there are no small-index reducibles on the boundary faces broken along $S^3$ (and, in fact, no small-index instantons whatsoever), and is proved as in the $i \not\equiv 1 \mod 3$ case of Lemma \ref{lemma:g-bdry1}.

\begin{lemma}\label{lemma:h2}
Suppose $A$ is an instanton with irreducible outer limits on $(W^i_{i-3}, \hat c^i_{i-3})$ or $(W^i_{i-3}, \check c^i_{i-3})$ with respect to one of the metrics in $G^i_{i-3}$ broken along $S^3$. Then $i(A) \ge 1$.
\end{lemma}

The third lemma asserts that there are no small-index reducibles on the boundary faces broken along the various $Z_j$, and that counting instantons on these boundary faces gives precisely the composite of the relevant instanton-counting maps, and is proved as in Lemma \ref{lemma:g-bdry0}.

\begin{lemma}\label{lemma:h3}
Suppose $A$ is an instanton with irreducible outer limits on $(W^i_{i-3}, \hat c^i_{i-3})$ or $(W^i_{i-3}, \check c^i_{i-3})$ with respect to one of the metrics in $G^i_{i-3}$ broken along $Z_{i-1}$ or $Z_{i-2}.$ Then $i(A) \ge -1$ with equality if and only if $A$ is an unbroken instanton with exactly two irreducible pieces.
\end{lemma}

Recall here that if $A$ is an instanton with respect to a broken metric $g$, we refer to $A$ as unbroken if it is composed of as many components as $g$ is, and we refer to $A$ as broken if it has strictly more components than $g$.

Stating the fourth lemma requires some preparation. For $i = 1, -1$, the family $G^i_{i-3}$ contains an interval of metrics $I^i_{\mathbb{RP}^3}$ broken along $\mathbb{RP}^3$; it restricts to a single metric on $N$, but an interval of metrics $I^i_{W}$ on $W^i_{i-3} \setminus N = \bar W^i_{i-3}$. The following lemma asserts that the chain-level map induced by $(W^i_{i-3}, I^i_{\mathbb{RP}^3})$ is equal to the chain-level map induced by $(\bar W^i_{i-3}, I^i_W)$. The proof is as in the $i \equiv 1 \mod 3$ case of Lemma \ref{lemma:g-bdry1}.

\begin{lemma}\label{lemma:h4}
Suppose $i \not\equiv 0 \mod 3$, and that $A$ is an instanton with irreducible outer limits on $(W^i_{i-3}, \hat c^i_{i-3})$ or $(W^i_{i-3}, \check c^i_{i-3})$ with respect to one of the metrics in $G^i_{i-3}$ broken along $\mathbb{RP}^3$. Then $i(A) \ge -1$, with equality if and only if $A$ is an unbroken instanton with exactly two pieces, the piece on $\bar W^i_{i-3}$ being an unbroken irreducible of index $-1$ and the piece on $N$ being the unique index $-1$ abelian connection. 
\end{lemma}

For the final face, consisting of metrics broken along $S^2 \times S^1$, the argument is somewhat different. We follow the argument of \cite[Section 6.3]{CDX}. 

Recall that $M^{i}_{i-3}\cong S^2 \times S^1$ divides $W^i_{i-3}$ into two pieces, one diffeomorphic to $V_i$ given by the complement of a regular neighborhood of $\{0\} \times K_i$ in $[-1,1] \times Z_i$  and the other diffeomorphic to the complement $N'$ of a neighborhood of an $S^1$ fiber in $N$. The corresponding one-parameter family of metrics is constant on $V_i$, and restricts to a family of metrics on $N'$ denoted by $I^i_{S^2 \times S^1}$. 

The character variety of flat $SU(2)$-connections on $S^2 \times S^1$ may be identified as $$\mathfrak X(S^2 \times S^1) \cong \text{Hom}\left(\pi_1(S^2 \times S^1), SU(2)\right)/\text{conj} \cong SU(2)/\text{conj} \cong [-1,1],$$
with the overall map given by sending $[A]$ to $\text{tr}(\text{Hol}_A(\{\ast\} \times S^1))/2$. The endpoint $1$ corresponds to the trivial connection, the endpoint $-1$ corresponds to the non-trivial central connection whose holonomy along $\{\ast\} \times S^1$ is $-I$, while the interior points of the interval correspond to abelian connections. 

Given a 4-manifold $W$ whose boundary is not necessarily a union of rational homology spheres, we define the index of an ASD connection over $W$ to be the index of the operator defined using weighted Sobolev spaces with a positive weight $\delta > 0$. For us, $W$ will have $S^2 \times S^1$ as its only component which is not a rational homology sphere. Restriction gives a map $M(W) \to \mathfrak X(S^2 \times S^1)$, and the index of an ASD connection computes the expected dimension of the fiber of this map.

The following statement can be proved similarly to those above, and is discussed for general $SU(N)$ following \cite[Remark 6.33]{CDX}.

\begin{lemma}\label{lemma:h5}
Suppose that $A$ is an instanton with irreducible outer limits on $(W^i_{i-3}, \hat c^i_{i-3})$ or $(W^i_{i-3}, \check c^i_{i-3})$ with respect to one of the metrics in $G^i_{i-3}$ broken along $S^2 \times S^1$. Then $i(A) \ge -1$, with equality if and only if $A$ is an unbroken instanton with exactly two pieces, the piece on $V_i$ being an unbroken irreducible of index $-1$ and the piece on $N'$ being a reducible instanton of index $-2$. 
\end{lemma}

That is, we may identify the moduli space of index $-1$ instantons $M(W^i_{i-3}, \hat c^i_{i-3}, I^i_{S^2 \times S^1}; \alpha, \alpha')_0$ with the fiber product $$M(V_i, \hat c^i_{i-3}; \alpha, \alpha')_0 \times_{\mathfrak X(S^2 \times S^1)} M(N', \hat c^i_{i-3}, I^i_{S^2 \times S^1})^{\textup{red}}_1,$$ and similarly with $\check c^i_{i-3}$. (The subscripts indicate the dimensions of the relevant spaces, not the indices of their constituent instantons.) It follows that there is a well-defined map defined by counting instantons supported by the family of metrics $(W^i_{i-3}, \hat c^i_{i-3}, I^i_{S^2 \times S^1})$ which have index equal to $-1$. A similar claim holds for $\check c^i_{i-3}$.

We thus have an understanding of the instantons on $\partial G^i_{i-3}$ of index $\le -1$. These put together, define the maps $h_i$ by the following formulae. We have $$h_i: C_*(Z_i) \to C_*(Z_i), \quad \quad h_i = \begin{cases} (W^i_{i-3}, G^i_{i-3}, \hat c^i_{i-3})_* + (W^i_{i-3}, G^i_{i-3}, \check c^i_{i-3})_* \quad i \in \{1, -1\} \\
(W^i_{i-3}, G^i_{i-3}, \hat c^i_{i-3})_* \oplus (W^i_{i-3}, G^i_{i-3}, \check c^i_{i-3})_* \quad i = 0.\end{cases}$$

It will be convenient to write $q_i: C_*(Z_i) \to C_*(Z_i)$ for the map defined by the same formula as above, but using the family of metrics $I^i_{S^2 \times S^1}$ in place of $G^i_{i-3}$.

\begin{proposition}
The maps $h_i$ written above are well-defined, and satisfy the relation $$d_{i-3} h_i + f_{i-2} g_i + g_{i-1} f_i + h_i d_i = q_i.$$
\end{proposition}
\begin{proof}
The combination of Lemmas \ref{lemma:h1}-\ref{lemma:h5} imply that the relevant moduli spaces for the definition of $h_i$ support only irreducible instantons of index $-2$ and no instantons of index $\le -3$, so the Uhlenbeck compactness theorem implies the relevant counts are finite. To investigate the relation, observe that Lemmas \ref{lemma:h1}-\ref{lemma:h5} imply that the interior of the moduli space of index $-1$ instantons is a smooth manifold of dimension $1$, and that these lemmas identify the boundaries of these moduli spaces. The boundary face in $G^i_{i-3}$ which breaks along $Z_{i-1}$ contributes $g_{i-1} f_i$; the boundary face which breaks along $Z_{i-2}$ contributes $f_{i-2} g_i$. The boundary faces which break along $S^3$ contribute the zero map. When $i = 0$, this leaves us only the map $q_0$ coming from the $S^2 \times S^1$ breakings. When $i \in \{1, -1\}$, both families $(W^i_{i-3}, G^i_{i-3}, \hat c^i_{i-3})$ and $(W^i_{i-3}, G^i_{i-3}, \check c^i_{i-3})$ also have a face corresponding to breaking along $\mathbb{RP}^3$. By Lemma \ref{lemma:h4}, these two faces induce the same map, so these terms cancel out upon adding the maps for $\check c^i_{i-3}$ and $\hat c^i_{i-3}$. This leaves only the $q_i$ terms.
\end{proof}

It remains to us to compute the map $q_i$.

\begin{proposition}\label{prop:h}
The map $q_i$ determined above is homotopic to the identity.
\end{proposition}
\begin{proof}
Again we follow the $N = 2$ case of \cite[Section 6.3]{CDX}. As depicted in Figure \ref{fig3}, the case $i = 0$ is simpler, as both boundary components of $I^i_{S^2 \times S^1}$ correspond to breakings along spheres. In this case \cite[Proposition 5.54]{CDX} identifies the moduli space $M(N', c, I^0_{S^2 \times S^1})^{\textup{red}}_1 \cong [-1, 1]$.  For $i \in \{1, -1\}$, one needs to take a little extra care in dealing with the $\mathbb{RP}^3$-broken metrics, and here it becomes essential that we work with both choices of geometric representative.

One can still identify $$M(N', \hat c^i_{i-3}, I^i_{S^2 \times S^1})^{\textup{red}}_1 \cong [-1, 0], \quad \quad M(N', \check c^i_{i-3}, I^i_{S^2 \times S^1})^{\textup{red}}_1 \cong [0, 1]$$ with $0$ corresponding to the metric broken along $\mathbb{RP}^3$ in both cases. These two metrics coincide. Furthermore, on the part of $N'$ bounded by $\mathbb{RP}^3$ and $S^2 \times S^1$, the two choices of bundle data coincide. It follows that on this piece of $N'$, the instantons associated to $0 \in [-1, 0]$ and $0 \in [0,1]$ coincide, and in particular so do their restrictions to $S^2 \times S^1$. It follows that these two moduli spaces $[-1,0]$ and $[0,1]$ can be pasted together along the broken metric to obtain a single moduli space homeomorphic to $[-1,1]$, equipped with a continuous map to $\mathfrak X(S^2 \times S^1) \cong [-1,1]$.

Now the argument proceeds the same for all $i$. That this map sends $\partial [-1,1]$ to $\partial \mathfrak X(S^2 \times S^1)$ identically --- and thus the map from $[-1,1]$ has degree $1$ in an appropriate sense --- is the $N = 2$ case of \cite[Proposition 5.53]{CDX}. This is established by an explicit computation \cite[Proposition 5.40]{CDX} for the broken metrics in $\partial [-1,1]$, ultimately because the restriction to $\mathfrak X(S^2 \times S^1)$ can be determined by a curvature integral. 

Thus, the map $q_i$ coincides with the map induced by the broken metric on $V_i \sqcup D^3 \times S^1$, as the natural map $M(D^3 \times S^1)^{\text{red}}_1 \to \mathfrak X(S^2 \times S^1)$ also has degree $1$. Choosing a path of metrics from the broken metric on $V_i \sqcup D^3 \times S^1$ to the product metric on $I \times Z_i$ gives rise to a homotopy from the map $q_i$ to the identity map. 
\end{proof}

\begin{proof}[Proof of Proposition \ref{prop:CDX-triangle}]
Using the triangle detection lemma, we have constructed an exact triangle whose vertices are of the expected form. We need to discuss the degree of the relevant maps and verify the given description of $f_1$ and $f_0$ as cobordism maps. The degree computation follows from the index formula (\ref{eqn:index}) and our understanding of the topology of the manifolds $W^i_{i-1}$. The map $f_1$ is by definition $$(W^1_0, \hat c^1_0)_* \oplus (W^1_0, \check c^1_0)_* = (W, \varnothing)_* \oplus (W, c_+)_*,$$ where $c_+$ is the cocore, as stated in Proposition \ref{prop:CDX-triangle}. The map $f_0$ is by definition $$(W^0_{-1}, \check c^1_0)_* \oplus -(W^0_{-1}, \check c^0_{-1})_* = (W', c_-)_* \oplus -(W', c_- \sqcup c_+)_*,$$ with $c_-$ the core and $c_+$ the cocore. Now observe that the induced map $(W', c)_*$ depends only on the relative homology class of $c$. Because $W'$ has intersection form $(-1)$ and $c_- \cap c_+$ is a single positively-oriented point, we see that $c_- = -c_+$ in $H^2(W'; \Bbb Z)$. Thus, $(W', c_- \sqcup c_+)_*$ coincides with $(W', \varnothing)_*$. This completes the proof of Proposition \ref{prop:CDX-triangle}.
\end{proof}

\subsection{Proof of Proposition \ref{prop:pm-one-map}}\label{sec:proof-37}
In the case $Y = Z_0 = S^3$, the chain complex $CI_*(S^3)$ is identically zero, and the maps $f_0$ and $f_1$ are trivial. This implies that the map $g_1: C(Z_1) \to C(Z_{-1})$ is a chain map because the composite $f_0 f_1$ factors through the trivial group, and we have $$d_{-1} g_1 + g_1 d_1 = f_0 f_1 = 0.$$ Similarly, because $f_0 g_2$ and $ g_0 f_1$ are trivial, the relations for $h_1$ and $h_{-1}$ imply that $f_{-1}$ and $g_1$ are chain homotopy inverse maps. In particular, $g_1$ induces an isomorphism on homology.

It suffices to explain why $g_1$ extends to an $IP$-morphism of degree $3$ and level $1/4 - \eta(K)$. The map $g_1$ is obtained by summing over two maps, both of which are obtained by counting instantons on $W^1_{-1}$ with respect to an appropriate family of metrics $G$ and $U(2)$-bundle. These $U(2)$-bundles correspond to geometric representatives $\hat c^1_{-1}, \check c^1_{-1} \subset W^1_{-1}$ for which $c^2 = -1$. Because $W^1_{-1}$ has rational homology sphere ends and $b_1(W) = b^+(W) = 0$, the degree and level of such a map is given by $$D = \dim G - 2c^2 = 1 +2  = 3, \quad \quad L = -c^2/4 - \eta(K) = 1/4 - \eta(K)$$ by a slight generalization of Proposition \ref{prop:IP-morphism} to the case that $W$ is equipped with a family of metrics. 

\section{Alexander polynomial constraints}
In this section, we give a proof of the following result.

\begin{theorem}\label{prop:alexander}
Suppose that $K$ is a knot in $S^3$ and $\pm 2$ is a cosmetic pair.  Then $\Delta_K = 1$.  
\end{theorem}

Let $K$ be a knot in a homology sphere $Y$.  Let $X_\pm (K)$ denote the unique 2-sheeted cover of $Y_{\pm 2}(K)$. If $\pm 2$ is a cosmetic pair for $K$, then $X_+(K)$ is orientation-preserving diffeomorphic to $X_-(K)$. We will use the following well-known surgery description of $X_\pm(K)$. Writing $\Sigma_2(K)$ for the branched double cover of $K$, the preimage $\widetilde K \subset \Sigma_2(K)$ remains null-homologous, so $\pm 1$ surgery on $\widetilde K$ makes sense. 
\begin{lemma}\label{lem:double-cover-surgery}
The manifold $X_\pm (K)$ is obtained from $\Sigma_2(K)$ by $\pm 1$ surgery on $\widetilde{K}$. 
\end{lemma}
\begin{proof}
Since $Y_{\pm 2}(K)$ is obtained from Dehn filling the exterior of $K$, the double cover is described by a suitable Dehn filling of a double cover of $Y - K$.  Since $H_1(Y - K)$ surjects onto $H_1(Y_{\pm 2}(K))$, we are Dehn filling the non-trivial double cover.  The slope is simply the lift of the slope downstairs, which is $\pm 2\mu_K + \lambda_K$.  The preimage of $\mu_K$ is $\mu_{\widetilde{K}}$ and the preimage of $\lambda_K$ is two copies of $\lambda_{\widetilde{K}}$.  Hence, the preimage of $\pm 2 \mu_K + \lambda_K$ is two copies of $\pm \mu_{\widetilde{K}} + \lambda_{\widetilde{K}}$ and the result follows.
\end{proof}

\begin{proof}[Proof of Theorem~\ref{prop:alexander}]
Let $K$ be a non-trivial knot in $S^3$ for which $\pm 2$ is a cosmetic pair. We have $g(K) = 2$ by \cite[Theorem 2(ii)]{Hanselman}, so the Alexander polynomial takes the form
\[
\Delta_K(t) = at^2 + bt + c + bt^{-1} + at^{-2},  
\]
with $a, b$ possibly zero. We will also use that the preimage $\widetilde K$ has Alexander polynomial
\[
\Delta_{\widetilde{K}}(t) = \Delta_K(t^{1/2})\Delta_K(-t^{1/2}) = a^2 t^2 + (2ac - b^2) t + (2a^2 - 2b^2 + c^2) + (2ac - b^2) t^{-1} + a^2 t^{-2};
\] 
see e.g. \cite[Corollary 4.2]{HKL}. Now \cite[p182, Consequence 2]{BoyerLines} states that whenever a knot $J \subset Y$ admits cosmetic surgeries, its Alexander polynomial has $\Delta_J''(1) = 0$. Applying this to $K$ gives $8a + 2b = 0$ and so $b = -4a$. Since $|\Delta_K(1)| = 1$, we also have $c =6a \pm 1$.

Next, we study the Casson--Walker invariants of $X_\pm(K)$. Appealing to Lemma~\ref{lem:double-cover-surgery} and \cite[p182, Consequence 2]{BoyerLines} once more, we see that $\Delta''_{\widetilde{K}}(1) = 0$ as well. As 
\[
\Delta''_{\widetilde K}(1) = 2a^2 + 2(2ac - b^2) + 6a^2 = 8a^2 + 4a(6a \pm 1) - 2(-4a)^2 = \pm 4a,
\]
so combining these means $a = 0$ and hence $\Delta_K = 1$.  
\end{proof}

\bibliography{bib.bib}
\bibliographystyle{alpha}

\end{document}